\documentclass[11pt]{amsart}
\usepackage[utf8]{inputenc}
\usepackage{indentfirst,csquotes}
\usepackage{amssymb}

\usepackage[english]{babel}

\usepackage{fancyhdr} 
\usepackage{amsthm}
\usepackage{amssymb}
\usepackage{mathtools}
\usepackage{amsmath}
\usepackage{latexsym}
\usepackage{esint}
\usepackage{wasysym}
\usepackage{pst-eucl}
\usepackage{comment}
\usepackage{hyperref}
\usepackage{tikz}
\usepackage{orcidlink}
\usepackage{mathrsfs}
\usepackage{blkarray}
\hypersetup{
	colorlinks=true,
	linkcolor=blue,
	citecolor=teal
}

\newcommand{\gen}[1]{\{#1^k\}_{k\in \N}}
\newcommand{\dist}{{\rm dist}}

\newtheorem{Theorem}{Theorem}
\numberwithin{Theorem}{section}
\newtheorem{Proposition}[Theorem]{Proposition}
\newtheorem{Lemma}[Theorem]{Lemma}
\newtheorem{Corollary}[Theorem]{Corollary}
\theoremstyle{definition}
\newtheorem{Definition}[Theorem]{Definition}
\theoremstyle{remark}
\newtheorem{Remark}[Theorem]{Remark}
\numberwithin{equation}{section}

\makeatletter
\newenvironment{customproof}[1]{%
  \par\pushQED{\qed}%
  \normalfont\topsep6pt \trivlist
  \item[\hskip\labelsep\itshape
    Proof of #1\@addpunct{.}]\ignorespaces
}{%
  \popQED\endtrivlist\@endpefalse
}
\makeatother

\DeclareMathOperator{\R}{\mathbb{R}}
\DeclareMathOperator{\N}{\mathbb{N}}

\newcommand{\E}{\mathcal{E}}
\newcommand{\Ll}{\mathcal L}

\newcommand{\EsL}{\widetilde{\mathcal E}_{s,\Lambda}}

\newcommand{\eee}{\ensuremath{\varepsilon}}

\newcommand{\Ha}{{\mathcal{H}}}

\subjclass[2020]{Primary 49Q20, 49K40; Secondary 28A75, 49Q10.}
\keywords{Fractional perimeter, quantitative isoperimetric inequality, quantitative Cheeger inequality, generalized minimizers.}

\begin{document}
	\title[A strong quantitative form of the fractional isoperimetric inequality]{ A strong quantitative form of the fractional isoperimetric inequality}
	\author[Eleonora Cinti, Enzo Maria Merlino and Berardo Ruffini]{Eleonora Cinti$^{\orcidlink{0000-0001-8266-8204}}$, Enzo Maria Merlino$^{\orcidlink{0000-0001-8501-9613}}$ and Berardo Ruffini$^{\orcidlink{0000-0001-6434-6383}}$}
	\date{\today}
	\address{Eleonora Cinti \orcidlink{0000-0001-8266-8204}: Dipartimento di Matematica, Universit\`a di Bologna, Piazza di Porta \\ S.Donato 5, 40126, Bologna-Italy}
	\email{eleonora.cinti5@unibo.it}
    
    \address{Enzo Maria Merlino \orcidlink{0000-0001-8501-9613}: Dipartimento di Matematica, Universit\`a di Bologna, Piazza di Porta \\ S.Donato 5, 40126, Bologna-Italy}
	\email{enzomaria.merlino2@unibo.it}

    \address{Berardo Ruffini \orcidlink{0000-0001-6434-6383}: Dipartimento di Matematica, Universit\`a di Bologna, Piazza di Porta \\ S.Donato 5, 40126, Bologna-Italy}
	\email{berardo.ruffini@unibo.it}

\begin{abstract}
	We show a strong version of the fractional quantitative isoperimetric inequality, {in which the isoperimetric deficit controls not only the Fraenkel asymmetry but also a sort of oscillation of
	the boundary}. This generalizes the local result by Fusco and Julin in \cite{FJ}. The proof follows a regularization process as in \cite{FJ}  but it is quite different in its spirit. Then, as a consequence of the quantitative inequality, we prove some stability estimates for a fractional Cheeger inequality.

\end{abstract}

\maketitle
	
	\section{Introduction and main results}
\subsection{Background: local and nonlocal quantitative isoperimetric inequalities}	
	In recent years, interest in studying the stability of isoperimetric-type inequalities has been steadily increasing. After  early contributions dating back to the beginning of the last century (see, e.g., \cite{Ber, Bon}), the first quantitative version of the isoperimetric inequality in arbitrary dimension was proved by Fuglede in \cite{F}. He showed that if $E\subset\R^n$ is a \emph{nearly spherical set}, i.e., a Lipschitz regular set with barycenter at the origin and volume equal to that of the unit ball $B_1$, such that
	\begin{equation}\label{one}
		\partial E = \{ z(1 + u(z)) : z \in \partial B_1 \},
	\end{equation}
	with $\|u\|_{W^{1,\infty}}$ sufficiently small, then
	\begin{equation}\label{two}
		\|u\|^2_{W^{1,2}(\partial B_1)} \leq C \left( P(E) - P(B_1) \right),
	\end{equation}
for some constant $C>0$ depending only on  $n$.  Here $P(\cdot)$ denotes the perimeter in the sense of De Giorgi. From this estimate, Fuglede was able to deduce that the perimeter deficit $P(E) - P(B_1)$ also bounds from above the Hausdorff distance between $E$ and $B_1$, provided that $E$ is nearly spherical or convex.
	
	However, the Hausdorff distance proves to be a too strong notion when dealing with general sets of finite perimeter. Hall in \cite{Hall} replaced it by the so-called \emph{Fraenkel asymmetry} index
	$$
	\alpha(E) := \min_{y \in \mathbb{R}^n} \left\{ \frac{|E \Delta B_r(y)|}{|B_r|} : |B_r| = |E| \right\},
	$$
	obtaining a non-sharp stability estimate. A sharp quantitative inequality was later proved in \cite{FuMP}  by Fusco, Maggi, and Pratelli, who showed that there exists a constant $\gamma(n)$ such that for any set $E$ of finite perimeter and measure
	\begin{equation}\label{three}
		\alpha(E)^2 \leq \gamma(n) \, \delta(E),
	\end{equation}
	where $\delta(E)$ denotes the \emph{isoperimetric deficit}:
	$$
	\delta(E) := \frac{P(E) - P(B_r)}{P(B_r)}, \qquad \text{with } |B_r| = |E|.
	$$
	
	Right after the work \cite{FuMP}, which employed symmetrization techniques, several alternative approaches have been developed. In \cite{CL} another proof of \eqref{three} was proposed, based on the so-called \emph{Selection Principle}, which relies on the regularity theory of almost-minimizers for the perimeter. In \cite{FMP-inv} estimate \eqref{three} was generalized to the anisotropic perimeter via an optimal transportation approach, and in \cite{CGP+}, it was extended to the case of weighted perimeter in convex cones by employing ABP-type techniques.
	
	In \cite{FJ}, {Fusco and Julin improved the quantitative inequality \eqref{three}, by considering a stronger notion of asymmetry, which takes into account, not only the $L^1$ distance between $E$ and an optimal ball (namely, the Fraenkel asymmetry), but also the oscillation of the boundary.
	
	More precisely, given a set of finite perimeter $E$, let us denote by $\partial^*E$ the reduced boundary of $E$ and by $\nu_E(x)$ the measure theretic outer unit normal to $\partial^*E$ at $x$.  Given a ball $B_r(y)$ with the same volume as $E$, for every point $x\in \partial^* E$ ,  let
	$$
	\pi_{y,r}(x):=y+r\frac{x-y}{|x-y|},\qquad\text{for all $x\not=y$},
	$$
be	the projection of $x$ on $\partial B_r(y)$. Fusco and Julin introduced the asymmetry index given by
    {\begin{equation*}
	A(E):=\min_{y\in\R^n}\biggl\{\frac{|E\Delta B_r(y)|}{|B_r|}+\biggl(\frac{1}{2P(B_r)}\int_{\partial^*E}|\nu_E(x)-\nu_{B_r(y)}(\pi_{y,r}(x))|^2d\Ha^{n-1}(x)\biggr)^{1/2}: |B_r|=|E|\biggr\},
	\end{equation*}
	and they proved that there exists a constant $C(n)$, such that for every set $E\subset\R^n$ of finite perimeter
    \begin{equation}\label{mainFJ}
        A(E)^2\leq C(n)\delta(E).
	\end{equation}
	As remarked in \cite{FJ}, the estimate \eqref{mainFJ} is equivalent to  the estimate \eqref{two} for nearly spherical sets (see also (5.16) in \cite{F}) and thus it generalizes Fuglede’s estimate to all sets of finite perimeter. In order to prove \eqref{mainFJ}, they first observed that  the following Poincar\'e-type inequality holds true (see \cite[Proposition 1.2]{FJ})}
\begin{equation}\label{poincare}
A(E) + \delta(E)^{1/2} \leq c\,\beta(E),
\end{equation}
for some $c=c(n)>0$, where
\begin{equation}\label{defbeta}
\beta(E) := \min_{y \in \mathbb{R}^n} \biggl\{ \left( \frac{1}{2P(B_r)} \int_{\partial^* E} \left| \nu_E(x) - \nu_{B_r(y)}(\pi_{y,r}(x)) \right|^2 d\mathcal{H}^{n-1}(x) \right)^{1/2} : |B_r| = |E| \biggr\}.
\end{equation}
Thus, the proof of \eqref{mainFJ}, reduces to show that
\begin{equation*}
\beta(E)^2 \leq C\,\delta(E),
\end{equation*}
for a suitable constant $C=C(n)$.}

For further works concerning strong-type forms of the quantitative isoperimetric inequality, we refer, for instance, to \cite{BDF,BogDuzSch} for different geometric settings, to \cite{N,DeMas} for the anisotropic perimeter, to \cite{BBJ} for the Gaussian setting, and to the recent contribution \cite{CPP} addressing as well capillarity-type energies.

In this paper we  investigate the nonlocal counterpart of the quantitative inequality \eqref{mainFJ}.
	
	The notion of fractional perimeter, introduced in \cite{CRS}, is a variant of the classical notion of perimeter in the sense of De Giorgi, which takes into account long-range interactions between sets. Precisely, given $s\in(0,1)$, the $s$-perimeter of a measurable set $E\subset\mathbb{R}^n$ is defined as
	\begin{equation}\label{eq:ps}
	P_s(E):=\int_E\int_{E^c}\frac{dx\,dy}{|x-y|^{n+s}}=\frac{1}{2} \left[\chi_E\right]_{W^{s,1}(\mathbb{R}^n)},
	\end{equation}
	where $\chi_E$ denotes the characteristic function of $E$ and $[\cdot]_{W^{s,1}(\mathbb{R}^n)}$ denotes the Gagliardo seminorm in the fractional Sobolev space $W^{s,1}(\mathbb{R}^n)$. 

{The isoperimetric-type inequality for the $s$-perimeters, shown in \cite{FS2008},  states that for any measurable set $E\subseteq \R^n$ with $|E|<\infty$, we have 
    \begin{equation}\label{non-isop-in}
		\frac{P_s(E)}{|E|^{(n-s)/n}}\ge \frac{P_s(B_1)}{|B_1|^{(n-s)/n}},
	\end{equation}
    where equality holds if and only if $E$ is a ball.}
 
A non-sharp quantitative version of \eqref{non-isop-in} was first established in \cite{FMM2011}, where the authors proved that
	\begin{equation*}
		\alpha(E)^{4/s}\leq C(n,s) \delta_s(E),
	\end{equation*}
where
	\begin{equation*}
		\delta_s(E) := \frac{P_s(E) - P_s(B_r)}{P_s(B_r)}, \qquad \text{with } |B_r| = |E|,
	\end{equation*}
	is the $s$-isoperimetric deficit. Its sharp improvement was later obtained in \cite{F2M3}: for every $n\geq 2$ and $s_0\in (0,1)$, there exists a {positive constant $c_Q>0$, depending only on $n$ and $s_0$,} such that
	\begin{equation}\label{isop-F2M3}
		\alpha(E)^2\leq c_Q\,\delta_s(E),
	\end{equation}
	for $s\in[s_0,1]$. It is worth stressing that estimate \eqref{isop-F2M3} holds uniformly with respect to $s$, as long as $s$ is bounded away from zero. Therefore, thanks to the convergence results established in \cite{BBM}, it also implies, in particular, the quantitative isoperimetric inequality for the classical perimeter as a limiting case.
	
	As in \cite{CL}, the starting point to prove \eqref{isop-F2M3} is a Fuglede-type result  stating that there exist two positive constants $\eta_0$ and $c_0$, depending only on $n$, such that if $E$ is a nearly spherical set as in \eqref{one}, with $\|u\|_{C^{1}(\partial B_1)}<\eta_0$, then for all $s\in(0,1)$,
	\begin{equation}\label{fugF2M3}
		P_s(E) - P_s(B_1) \geq c_0 \left( [u]_{\frac{1+s}{2}}^2 + s\,P_s(B_1)\,\|u\|_{L^2(\partial B_1)}^2 \right),
	\end{equation}
	where the Gagliardo seminorm $[u]_{\frac{1+s}{2}}$ is defined by
	\[
	[u]_{\frac{1+s}{2}}^2 := [u]_{H^{\frac{1+s}{2}}(\partial B_1)}^2 = \int_{\partial B_1}\int_{\partial B_1} \frac{|u(x)-u(y)|^2}{|x-y|^{n+s}}\,d\mathcal{H}^{n-1}(x)\,d\mathcal{H}^{n-1}(y).
	\]
	By the asymptotic result of \cite[(8.4)]{F2M3}, the inequality \eqref{fugF2M3} recovers Fuglede’s estimate \eqref{two} in the limit as $s\to 1$.

\subsection{Main results}
	The scope of this work is to derive a nonlocal quantitative isoperimetric inequality that fully exploits the presence of the complete fractional Sobolev norm appearing in \eqref{fugF2M3}. In pursuing this goal, the first difficulty to keep in mind is that, since a set with finite $s$-perimeter is only measurable, an asymmetry index based on the oscillation of the normals may not be well defined. Let us take a closer look at the index defined in \eqref{defbeta}. As  observed in \cite{FJ}, by applying the divergence theorem, one finds that
	\[
	\begin{split}
		\frac{1}{2}\int_{\partial^*E}|\nu_E(x)-\nu_{B_r(y)}(\pi_{y,r}(x))|^2\,d\mathcal{H}^{n-1}(x) & = \int_{\partial^*E} \left(1-\nu_E(x)\cdot\frac{x-y}{|x-y|}\right)\,d\mathcal{H}^{n-1}(x) \\
		& = P(E) - \int_E \frac{n-1}{|x-y|}\,dx\,,
	\end{split}
	\]
	where, as before, $\pi_{y,r}$ denotes the radial projection onto $\partial B_r(y)$. Therefore, we can write
	\begin{equation}
		\label{remember}
		\beta(E)^2 = \frac{1}{P(B_r)} \left(P(E) - (n-1) V(E)\right),\qquad \text{with } |B_r| = |E|,\end{equation}
	where we have set
	\begin{equation*}
		V(E) := \max_{y\in\R^n} \int_E \frac{1}{|x-y|} \, dx\,.
	\end{equation*}
	With this motivation in mind, in the nonlocal setting, fixed $s\in (0,1)$, and a measurable set $E$, we define
	\begin{equation}\label{defbetas}
		\beta_s(E):=\left(\frac{1}{P_s(B_r)}\left( P_s(E)-c_{n,s}V_s(E)\right) \right)^{1/2 } \qquad \text{with } |B_r| = |E|,\end{equation}
	where
	\begin{equation}\label{defV}
		V_s(E):=\max_{y\in\mathbb{R}^n}\int_{E}\frac{1}{|x-y|^s}\,dx
	\end{equation}
	and $c_{n,s}:=\frac{P_s(B_1)(n-s)}{n\omega_n}$ is the normalization constant that ensures that $\beta_s(B)=0$, for every ball $B$. As we will  stress later on, by the fractional isoperimetric inequality and by the Riesz rearrangement inequality, altogether with their rigidity cases, one can see that $\beta_s(E)>0$ for all sets $E$ which are not balls. {Consequently, in analogy with the local case, we define the following asymmetry index
    \begin{equation}
    A_s(E):=\min_{y\in\R^n}\biggl\{\frac{|E\Delta B_r(y)|}{|B_r|}+\biggl(\frac{1}{P_s(B_r)}\left( P_s(E)-c_{n,s}\int_E\frac{1}{|x-y|^s}\,dx\right)\biggr)^{1/2}: |B_r|=|E|\biggr\}.
    \end{equation}}
    
    Then, our main result reads as follows.
	
	\begin{Theorem}\label{our-quant-beta}
		For every $n\ge 2$ and $s\in(0,1)$ there exists a positive constant $c(n,s)$ such that
		\begin{equation}\label{eq:our-quant-beta}
    A_s(E)^2\leq C_0(n,s)\delta_s(E)			
		\end{equation}
		whenever $0<|E|<\infty$.
	\end{Theorem}
	
{ 

We will observe, in Remark \ref{rmk-asy}, that the constant $C_0(n,s)$ appearing in the above quantitative estimate is uniformly bounded as $s\rightarrow 1^-$. Hence, our main result recovers the analog estimate in the classical local setting.

Some comments on the proof are in order. 

In the local case treated in \cite{FJ}, the main strategy is to reduce the problem to the case of nearly spherical sets, so as to exploit Fuglede’s inequality \eqref{two}. This typical approach, introduced in \cite{CL}, is generally based on an appropriate notion of perimeter almost minimizers. However, in addition to using the regularity theory for perimeter almost minimizers, the authors in \cite{FJ} also rely on properties of perimeter quasi-minimizers. This latter relaxed notion of perimeter minimizers is the natural counterpart, within the theory of sets of finite perimeter, of the notion of quasi-minima introduced by Giaquinta and Giusti in \cite{GG} for variational integrals. Although quasi-minimality is a considerably weaker condition than almost minimality, it still guarantees a mild degree of regularity. In particular, quasi-minimizers are uniformly porous (see, e.g., \cite{DS,KKLS}), a property which plays a key role in upgrading convergence in $L^1$ to convergence in the Hausdorff sense. This strategy does not appear to lend itself to a clear or straightforward extension to the nonlocal setting. Therefore, in this work, we adopt a slightly different approach, which we present in the context of the $s$-perimeter, though it is equally applicable to the classical perimeter. Let us summarize the main ideas underlying our method.

First, we aim to establish the following Poincaré-type inequality:
\begin{equation}\label{s-poincare}
A_s(E)+\delta_s(E)^{1/2}\leq C \beta_s(E)
\end{equation}
for some constant $C = C(n,s) > 0$. This inequality represents the fractional counterpart of \eqref{poincare}. In light of \eqref{s-poincare}, the proof of the sharp quantitative estimate \eqref{eq:our-quant-beta} reduces to showing that
\begin{equation}\label{redution1}
\beta_s(E)^2\leq C \delta_s(E),
\end{equation}
for a suitable constant $C = C(n,s) > 0$.

Next, recalling the definitions of $\beta_s$ and $V_s$, we observe that
\begin{equation}\label{use-def-beta}
\begin{split}
{\beta^2_s(E)} &= \frac{1}{P_s(B_r)}\left( P_s(E)-c_{n,s}V_s(E)\right)\\
&= \frac{1}{P_s(B_r)}\left( P_s(E)-P_s(B_r) + P_s(B_r) - c_{n,s}V_s(E)\right)\\
&= \delta_s(E) + \frac{c_{n,s}}{P_s(B_r)}(V_s(B_r)-V_s(E)) = {\delta_s(E) + \zeta_s(E)}.
\end{split}
\end{equation}
where we have set
\begin{equation}\label{def-zeta-s}
\zeta_s(E):=\frac{c_{n,s}}{P_s(B_r)}(V_s(B_r)-V_s(E))=\frac{V_s(B_r)-V_s(E)}{V_s(B_r)}.
\end{equation}

Hence, the proof of \eqref{redution1} is reduced to demonstrating that there exists a constant $C_{n,s} > 0$ such that
\begin{equation}\label{our-quant-zeta}
C_{n,s}\zeta_s(E) \le \delta_s(E).
\end{equation}

Since both $\zeta_s(E)$ and $\delta_s(E)$ are scale invariant by definition, it suffices to establish \eqref{our-quant-zeta} under the normalization $|E| = \omega_n$, and thus we may take $B_r = B_1$. That is, we reduce the problem to prove that
$$C_{n,s}(V_s(B_1)-V_s(E))\le P_s(E)-P_s(B_1),$$
which is equivalent to 
$$P_s(B_1)+C_{n,s}V_s(B_1)\le P_s(E)+C_{n,s}V_s(E),$$
for some constant $C_{n,s} > 0$.

At this point, the main idea, inspired by the approach in \cite{DiCNRV}, is to rephrase the problem in terms of a { minimization problem for a new variational functional}. Specifically, we seek for a quantity $\varepsilon_1 > 0$, depending only on $n$ and $s$, such that for every $\varepsilon \le \varepsilon_1$, the unit ball is the unique minimizer of the energy functional
\begin{equation}\label{infclass}
\min_{|E|=\omega_n} \left\{\mathcal{E}_s(E):= P_s(E) + \varepsilon \,V_s(E) \right\}.
\end{equation}

{We stress that the energy functional $\mathcal E_s$, which is the main object of study of this work, depends on the parameter $\varepsilon$, but we do not write explicitly such dependence just for simplicity of notation}. 

In order to carry out our strategy, we first prove the existence of minimizers for \eqref{infclass} via a generalized formulation inspired by the concentration-compactness principle of Lions \cite{Lio}, recently adapted to the nonlocal setting (see, for instance, \cite{DiCNRV, CT, GNR4, NO}). This approach, combined with the stability estimate \eqref{isop-F2M3}, allows us to establish that minimizers of \eqref{infclass} are close to a ball in the $L^1$-topology.

To show that balls are indeed the unique (rigid) minimizers, we then exploit certain regularity properties of $s$-perimeter almost minimizers}. More precisely, we need to argue in two steps. In a first step, we show that minimizers of the energy functional in \eqref{infclass} are indeed \emph{$(\rho, R)$-minimizers} (or almost-minimizers with modulus of continuity $\rho$, as defined, for example, in \cite[Definition 1.3]{CT}), where $\rho$ is of the form $\rho(r)=\varepsilon C r^{n-s}$. Such modulus of continuity has the critical behaviour $r^{n-s}$ (and hence has the same scaling of the fractional perimeter), which does not imply $C^{1,\alpha}$-regularity. Nevertheless, if the coefficient $\varepsilon$ is small enough, as recently pointed out in \cite{CT},
such almost minimizers satisfy both interior and exterior density estimates.

This density property is a crucial tool for upgrading $L^1$ convergence to Hausdorff convergence. This, in turn, allows us to prove that minimizers of \eqref{infclass} are actually $\Lambda$-minimizers for the $s$-perimeter, in the sense of \cite{F2M3}. Combining this $\Lambda$-minimality property with \cite[Corollary 3.6]{F2M3}, we finally get that, for $\varepsilon$ small enough, minimizers of \eqref{infclass} are nearly spherical sets.

A final step, based on a Fuglede-type argument, then yields the rigidity of balls as minimizers for \eqref{infclass}. A further remarkable consequence of the last step is that, for nearly spherical sets, the asymmetry index $A_s$ introduced here is essentially equivalent to the fractional Sobolev norm $H^{\frac{1+s}{2}}(\partial B_1)$ appearing on the right-hand side of the Fuglede-type stability estimate \eqref{fugF2M3} proved in \cite{F2M3}, see Remark \ref{rmk:fuglede}. This highlights that our choice of the fractional oscillation index $\beta_s$ is not merely motivated by analogy with the classical (local) setting, as in \eqref{remember}, but rather is tailored to fully exploit the structure of the complete fractional Sobolev norm featured in \eqref{fugF2M3}, which was precisely our original purpose.

\vspace{0.1cm}
Eventually, as an application of our stability estimate, we derive a quantitative bound for the fractional Cheeger constant, defined as follows.
{	Let $\Omega \subset \mathbb R^n$ be an open set with finite measure and let $m>\frac{n-s}{n}$. The \textit{fractional Cheeger constant} (or \textit{$(m,s)$-Cheeger constant}) of $\Omega$ is defined by
$$
h_{m,s}(\Omega):=\inf\left\{\frac{P_s(E)}{|E|^m},\;\;E\subset \Omega\right\}.$$}

We establish the following quantitative stability estimate. 
	\begin{Theorem}\label{Thm:Ch-2}
		Let $\Omega \subset \mathbb R^n$ be an open set with finite measure and let $m> \frac{n-s}{n}.$ Then, there exists a constant $\kappa({m,n,s})>0$ such that
		\begin{equation}\label{quant-cheeger-2}
        \frac{h_{m,s}(\Omega) - h_{m,s}(B_r)}{h_{m,s}(B_r)} \ge  \kappa(m,n,s)\,\zeta_s(\Omega),\qquad \text{with } |B_r|=|\Omega|,		\end{equation}
        where $\zeta_s(\Omega)$ is the asymmetry index defined in \eqref{def-zeta-s}.
	\end{Theorem}
{ Furthermore, we discuss the failure of analogous estimates involving the oscillation index $\beta_s$ (or, equivalently, $A_s$), providing explicit counterexamples.}

    The paper is organized as follows. In Section \ref{sec:prel}, we introduce the main notations used throughout the paper and collect some preliminary results that will be needed in the sequel. In Sections \ref{sec:exgen} and \ref{sec:exmmin}, we first show the existence of generalized minimizers then that they are indeed classical minimizers. {Section \ref{sec:regularity} contains the reduction procedure to nearly spherical sets. In Section \ref{sec:rigidity} we finally show that such minimizers correspond to the ball, and we use this result to deduce Theorem \ref{our-quant-beta}. {Finally, in Section \ref{sec:cheeger}, we exploit our stability estimate to establish a quantitative result for fractional Cheeger constants, and we provide explicit counterexamples showing the failure of analogous bounds involving the index $\beta_s$.}


	\section{Preliminary notions and results}\label{sec:prel}
	In this section we introduce some preliminary properties of the energies involved in the paper. We begin with some elementary properties of the fractional perimeter defined in \eqref{eq:ps}.
	
	\begin{Proposition}\label{elementary_properties_ps}
		Let $E\subset\R^n$ be a measurable set.
		\begin{itemize}
			\item[(\emph{i})] (Subadditivity)$\ $ For any $F\subset\R^n$ such that $|E\cap F|=0$, we have
			\begin{equation*}
				P_s(E\cup F)= P_s(E)+P_s(F)-2\int_{E}\int_F\frac{dxdy}{|x-y|^{n+s}}.
			\end{equation*}
			In particular 
			\begin{equation*}
				P_s(E\cup F)\leq P_s(E)+P_s(F).
			\end{equation*}
			
			\item[(\emph{ii})] (Translation invariance)$\ $ Let $x\in\R^n$. Then
			\begin{equation*}
				P_s(E+x)=P_s(E).
			\end{equation*}
			
			\item[(\emph{iii})] (Scaling)$\ $ Let $\lambda>0$. Then
			\begin{equation*}
				P_s(\lambda E)=\lambda^{n-s}P_s(E).
			\end{equation*}

            \item[(\emph{iv})]
			\begin{equation*}
				P_s(E\cap B_\rho)\to P_s(E), \quad \text{ as }\rho\to+\infty.
			\end{equation*}		\end{itemize}
	\end{Proposition}
    \begin{proof}
    {Points (\emph{i}), (\emph{ii}) and (\emph{iii}) are direct consequences of the definition of $P_s$. For (\emph{iv}), we observe that by \cite[Lemma B.1]{F2M3}, for any $\rho>0$, $P_s(E\cap B_\rho)\leq P_s(E)$. Hence, since $E\cap B_\rho\to E$ in $L^1$ as $\rho\to+\infty$, by lower semicontinuity of $P_s$, we have
        $$ P_s(E)\leq \liminf_{\rho\to+\infty}P_s(E\cap B_\rho)\leq P_s(E) $$
        which implies the claim.}
    \end{proof}

We pass now to some properties of $V_s$, defined in \eqref{defV}. 
First, we observe that by a simple compactness argument, it is possible to prove that the maximum in \eqref{defV} is achieved, but it may not be unique. If $y\in \R^n$ is such that 
$$V_s(E) = \int_E \frac{dx}{|x-y|^s},$$
then we call $y$ a \textit{$V_s$-center of $E$}, and we denote by $y_E$ a generic $V_s$-center of $E$.

An important feature is that $V_s$ is maximed by balls under volume constraint. Let us recall some notations in order to see that.}

{ Given a measurable set $E\subset\R^n$ of finite measure, we consider its symmetric rearrangement $E^*$, i.e.  the open ball of measure $|E|$ centered at the origin. Similarly, given a nonnegative measurable function $f:\R^n\to[0,+\infty)$ vanishing at infinity, its symmetric decreasing rearrangement is defined as 
	\[
	f^*(x):=\int_0^\infty\chi_{\{f>t\}^*}(x)\,dt,
    	\]
    where $\chi_A$ denotes the charachteristic function of the set $A$.

	Note that such a function is measurable and lower semicontinuous. We recall (see, e.g., \cite[Theorem 3.4]{LL}) that for any two functions $f,g:\R^n\to[0,+\infty)$ vanishing at infinity, then the following concentration inequality holds
	\begin{equation}\label{eq:rearrineq}
	\int_{\R^n}f(x)g(x)\, dx\le 	\int_{\R^n}f^*(x)g^*(x)\,dx.
	\end{equation}
	Moreover, if $f$ is a radially decreasing symmetric function (i.e. $f(x)=\widetilde{f}(|x|)$ for some decreasing $\widetilde{f}:\R\to[0,+\infty)$), then equality in \eqref{eq:rearrineq} holds if and only if  $g=g^*$, that is $g$ is radially decreasing symmetric as well.
	
		 For any $y\in\R^n$, by applying \eqref{eq:rearrineq} with $f(x)=1/|x|^s$ and $g=\chi_{E-y}$ for a measurable set $E\subset\R^n$,  we get that 
        \begin{equation}\label{rear-1} \int_E \frac{dx}{|x-y|^s}=\int_{\R^n}\frac{1}{|x|^s}\,\chi_{E-y}(x)\,dx\leq \int_{\R^n}\frac{1}{|x|^s}\,\chi_{E^*}(x)\,dx=\int_{E^*}\frac{dx}{|x|^s}
        \end{equation}
         which implies
		\begin{equation}\label{rear-ineq}
			V_s(E) \leq V_s(E^*),
		\end{equation}
	where equality holds only if $E$ is a ball. Notice also that a similar argument proves that any ball has as a unique $V_s-$center its center. In particular, we have
    \begin{equation}\label{Vs-balls}
        V_s(B_r)=\int_{B_r}\frac{dx}{|x|^s}=\frac{n\omega_n}{n-s}r^{n-s}.
\end{equation}}
	
	We state some useful properties of $V_s$.
	\begin{Proposition}[Properties of $V_s$]\label{elementary_properties_vs}
		Let $E\subset\R^n$ be a measurable set.
		\begin{itemize}
			\item[(\emph{i})] (Subadditivity)$\ $ For any $F\subset\R^n$ such that $|E\cap F|=0$, we have
			\begin{equation*}
				V_s(E\cup F)\leq V_s(E)+V_s(F);
			\end{equation*}
			
			\item[(\emph{ii})] (Translation invariance)$\ $ Let $x\in\R^n$. Then 
			\begin{equation*}
				V_s(E+x)=V_s(E).
			\end{equation*}
			
			\item[(\emph{iii})] (Scaling)$\ $ Let $\lambda>0$. Then 
			\begin{equation*}
				V_s(\lambda E)=\lambda^{n-s}V_s(E).
			\end{equation*}
			
			\item[(\emph{iv})](H\"older continuity with respect to $L^1$ convergence of sets). Given two measurable sets $E,F\subset\R^n$, we have
			\begin{equation*}
				|V_s(E)-V_s(F)|\leq \frac{n\omega_n^{s/n}}{n-s}|E\Delta F|^{\frac{n-s}{n}}.
			\end{equation*}
		\end{itemize}
	\end{Proposition}
	\begin{proof}
		The subadditivity follows straightforwardly from the definition of $V_s$. Let us prove (\emph{ii}). Fixed $x\in\R^n$, then 
		$$ V_s(E+x)=\max_{y\in\R^n}\int_{E+x}\frac{d\xi}{|\xi-y|^s}=\max_{y\in\R^n}\int_{E}\frac{d\eta}{|\eta-(y-x)|^s}=\max_{y\in\R^n}\int_{E}\frac{d\eta}{|\eta-y|^s}=V_s(E).$$\\
		Similarly, regarding (\emph{iii}), fixed $\lambda>0$, we have
		\begin{equation*}
			V_s(\lambda\,E) = \max_{y\in\mathbb{R}^n}\int_{\lambda E}\frac{dx}{|x-y|^s}= \max_{y\in\mathbb{R}^n}\int_{E}\frac{\lambda^{n-s}}{|x-\frac{y}{\lambda}|^s}\,dx= \lambda^{n-s}\max_{y\in\mathbb{R}^n}\int_{E}\frac{dx}{|x-y|^s}=\lambda^{n-s}V_s(E).
		\end{equation*}
		Finally, we prove (\emph{iv}). Let $y_E$ be a $V_s$-center of $E$, i.e $V_s(E)=\int_E\frac{dx}{|x-y_E|^s}$, then 
		\begin{equation}\label{vE-vF}
			V_s(E)-V_s(F)\leq \int_{E}\frac{dx}{|x-y_E|^s}-\int_{F}\frac{dx}{|x-y_E|^s}\leq \int_{E\Delta F}\frac{dx}{|x-y_F|^s}.
		\end{equation}
		Now, let $r>0$ be such that $(E\Delta F)^*=(E\Delta F+y_E)^*=B_r$. Then by \eqref{rear-1}, we get that { 
		\begin{equation} \label{stima-diffsimm}
			\int_{E\Delta F} \frac{dx }{|x-y_E|^s}  \le \int_{B_r} \frac{dx}{|x|^s} =\frac{n\omega_n}{n-s}r^{n-s}=\frac{n\omega_n^{ s/n}}{n-s}|E\Delta F|^{\frac{n-s}{n}}.       
		\end{equation}
		which, together with \eqref{vE-vF}, entails
		$$ V_s(E)-V_s(F)\leq
        \frac{n\omega_n^{s/n}}{n-s}|E\Delta F|^{\frac{n-s}{n}}.$$}
		By interchanging $E$ with $F$, we prove the same estimate for $V_s(F)-V_s(E)$, implying the claim.
	\end{proof}

{The following lemma ensures that if a set is close to the unit ball $B_1$ in the $L^1$-topology, then its $V_s$-center is close to the origin. This represents the analogous result of \cite[Lemma 3.1]{FJ} (see also \cite[Lemma 2.3]{N}) in our fractional setting.  
\begin{Lemma}\label{lemma-center}
For every $\eee>0$ there exists $\eta>0$ such that if $|F \Delta B_1| < \eta$, then $|y_F| < \eee$ for every $V_s$-center $y_F$ of $F$.
\end{Lemma}
\begin{proof}
Let us argue by contradiction, assuming that there exits a sequence of sets $\gen{F}$ such that $|F_k\Delta B_1|\rightarrow 0$ and $y_{F_k} \to y_0$ with $|y_0| \geq \eee$, for some $\eee >0$, as $k\to \infty$. By the triangular inequality, we have
\begin{equation}\label{eq:001}
\begin{split}
\int_{B_1} \frac{dx}{|x|^s}\leq  \bigg|\int_{B_1} \frac{dx}{|x|^s} - \int_{F_k} \frac{dx}{|x-y_{F_k}|^s} \bigg|+\bigg| \int_{F_k} \frac{dx}{|x -y_{F_k}|^s} &- \int_{B_1} \frac{dx}{|x -y_{F_k}|^s} \bigg|
+\int_{B_1} \frac{dx}{|x -y_{F_k}|^s}.
\end{split}  
\end{equation}
Letting $k \to \infty$, since $F_k \to B_1$ in $L^1$, by point (\emph{iv}) of Proposition \ref{elementary_properties_vs}, we get 
\begin{equation}\label{eq:002}
\bigg|\int_{B_1} \frac{dx}{|x|^s} - \int_{F_k} \frac{dx}{|x-y_{F_k}|^s} \bigg|=|V_s(B_1)-V_s(F_k)|\leq \frac{n\omega_n^{s/n}}{n-s}|B_1\Delta F_k|^{\frac{n-s}{n}}\to 0.
\end{equation}
Moreover, taking into account \eqref{stima-diffsimm}, we obtain 
\begin{equation} \label{eq:003}
\begin{split}
\bigg| \int_{F_k} \frac{dx}{|x -y_{F_k}|^s} - \int_{B_1} \frac{dx}{|x -y_{F_k}|^s} \bigg|&=\bigg| \int_{F_k\setminus B_1} \frac{dx}{|x -y_{F_k}|^s} - \int_{B_1\setminus F_k} \frac{dx}{|x -y_{F_k}|^s} \bigg|\\
& \leq \int_{F_k\setminus B_1} \frac{dx}{|x -y_{F_k}|^s} + \int_{B_1\setminus F_k} \frac{dx}{|x -y_{F_k}|^s}\\
& =\int_{F_k\Delta B_1} \frac{dx}{|x -y_{F_k}|^s}\leq \frac{n\omega_n^{ s/n}}{n-s}|F_k\Delta B_1|^{\frac{n-s}{n}}\to 0,
\end{split}
\end{equation}
as $k\to \infty$. On the other hand, since $B_1$ has as unique $V_s$-center the origin, for all $k\in\N$, we have
\begin{equation}\label{eq:004}
    \int_{B_1}\frac{dx}{|x-y_{F_k}|^s}\leq \int_{B_1} \frac{dx}{|x|^s}.
\end{equation}
Consequently, combining \eqref{eq:001},\eqref{eq:002}, \eqref{eq:003} and \eqref{eq:004}, we conclude that
\begin{equation*}
    \lim_{k\to \infty}\int_{B_1}\frac{dx}{|x -y_{F_k}|^s}=\int_{B_1} \frac{dx}{|x|^s},\end{equation*}
which is true only if $|y_{F_k}|\to 0$, contradicting the fact that $|y_0|\geq \eee>0$.
\end{proof}}    

{Finally, we conclude this section by proving a Poincaré-type inequality which ensures that the oscillation index $\beta_s$, defined in \eqref{defbetas}, controls the entire asymmetry $A_s$. This is the fractional analogue of \cite[Proposition 1.2]{FJ}.

	\begin{Proposition}\label{prop:s-poincare}
		Let $E\subset\R^n$ a measurable set such that $0<|E|<\infty$, then there exists a constant $C=C(n,s)$ such that 
{ 
\begin{equation}\label{eq:s-poincare}
    A_s(E)+\delta_s(E)^{1/2}\leq C \beta_s(E).	
		\end{equation}}
	\end{Proposition}
	\begin{proof} Since the quantities $A_s(E)$, $\beta_s(E)$, $\zeta_s(E)$ and $\delta_s(E)$ are scale invariant, we shall assume without loss of generality that $|E|=\omega_n$. Moreover, up to translations, we can also assume that $E$ is $V_s$-centered at the origin, i.e.
		\begin{equation*}
			V_s(E)=\int_{E}\frac{1}{|x|^s}\,dx.
		\end{equation*}
		Then, we have
		\[
		\begin{split}
			V_s(B_1)-V_s(E)&=  \int_{B_1} \frac{1}{|x|^s} \, dx -  \int_{E} \frac{1}{|x|^s} \, dx\\
			&= \int_{ B_1 \setminus E} \frac{1}{|x|^s} \, dx - \int_{E\setminus B_1} \frac{1}{|x|^s} \, dx,
		\end{split}
		\]
        which, recalling \eqref{use-def-beta}, implies
        \begin{equation}\label{new-betas}
        \beta_s^2(E)=\delta_s(E)+\frac{c_{n,s}}{P_s(B_1)}\left( \int_{ B_1 \setminus E} \frac{1}{|x|^s} \, dx - \int_{E\setminus B_1} \frac{1}{|x|^s} \, dx \right).        \end{equation}
		Since $|E| = |B_1|$ we have
		\begin{equation}
			\label{definition.a}
			|E \setminus B_1| =  |B_1 \setminus E| =:a<\omega_n.
		\end{equation}
        Let us denote by $A(R,1) = B_R \setminus B_1$ and $A(1,r) = B_1 \setminus B_r$  the two annuli such that $|A(R,1)| = |A(1, r)| =a$, where $a$ is defined in \eqref{definition.a}. In particular,
		\[
		R = \left( 1 + \frac{a}{\omega_n}\right)^{1/n} \qquad \text{and} \qquad r = \left( 1 - \frac{a}{\omega_n}\right)^{1/n}.
		\]   
        {We have that 
\begin{equation}\label{stima-anello}
 \int_{ A(R,1)} \frac{1}{|x|^s} \, dx- 		\int_{ E \setminus B_1} \frac{1}{|x|^s} \, dx = \int_{ A(R,1)\setminus(E\setminus B_1)} \frac{1}{|x|^s} \, dx- 		\int_{ (E \setminus B_1)\setminus A(R,1)} \frac{1}{|x|^s} \, dx.
 \end{equation}
		Let us denote by $U:=A(R,1)\setminus(E\setminus B_1)$ and $V:=(E \setminus B_1)\setminus A(R,1)$. Note that $|U|=|V|$ and that, if $f(x):=1/|x|^s$, then $f(x)\ge f(y)$ for all $(x,y)\in U\times V$. Consequently, by \eqref{stima-anello}, we get 
		\[
        \int_{ E \setminus B_1} \frac{1}{|x|^s} \, dx \leq \int_{ A(R,1)} \frac{1}{|x|^s} \, dx\,.		\]
 Arguing similarly, we also obtain 
		\[
		\int_{ B_1 \setminus E} \frac{1}{|x|^s} \, dx \geq  \int_{A(1, r)} \frac{1}{|x|^s} \, dx.
		\]}Thus, by \eqref{new-betas} and recalling that $c_{n,s}=\frac{P_s(B_1)(n-s)}{n\omega_n}$, we have
		\begin{equation}
			\label{long.calculation}
			\begin{split}
                \beta_s(E)^2&\geq\,\delta_s(E)+\frac{c_{n,s}}{P_s(B_1)}\left(\int_{ A(1,r)} \frac{1}{|x|^s} \, dx - \int_{ A(R,1)} \frac{1}{|x|^s} \, dx\right) \\
				&=\, \delta_s(E)+
                \bigl((1 - r^{n-s})-(R^{n-s} -1)\bigr) \\
				&=\, \delta_s(E)+ 
            \left(2 -   \left( 1 + \frac{a}{\omega_n}\right)^{\frac{n-s}{n}} -   \left( 1 - \frac{a}{\omega_n}\right)^{\frac{n-s}{n}}  \right).
			\end{split}
		\end{equation}
		Now we observe that the function $f(t) = (1 + t)^{\frac{n-s}{n}}$ is uniformly concave in $[-1,1]$, i.e., 
		\begin{equation}\label{uniformcon}
			\frac{1}{2}\left(f(t) + f(\tau) \right) \leq f \left(\frac{t}{2} + \frac{\tau}{2} \right) - \alpha_{n,s} |t-\tau|^2,
		\end{equation}
		for any $t,\tau \in [-1,1]$ and $\alpha_{n,s} := - \frac{1}{4} \left( \sup_{t \in (-1,1)} f''(t)\right)  =  \frac{s(n-s)}{4n^2}2^{\frac{-n-s}{n}} >0$. Therefore, recalling \eqref{definition.a} and taking $t=\frac{a}{\omega_n}$ and $\tau=-t$ in \eqref{uniformcon}, by \eqref{long.calculation}  we get 
		\[
        \beta^2_s(E)\geq \delta_s(E)+\frac{8\,\alpha_{n,s}}{\omega_n^2} \, a^2 = \delta_s(E)+ {\tilde \alpha}_{n,s} \, (|E \setminus B_1| +  |B_1 \setminus E|)^2,
		\]
        for some constant $\tilde \alpha_{n,s}>0$.
	Finally, since $|E \setminus B_1| +  |B_1 \setminus E| = |E \Delta B_1| $, we obtain 
		\[ 
        (1+\tilde \alpha_{n,s})\beta^2_s(E)\geq \delta_s(E)+{\tilde \alpha}_{n,s} (\beta_s(E)^2+  |E \Delta B_1| ^2) \geq \delta_s(E)+CA_s(E)^2,
		\]
	which implies the claim for a suitable constant $C$ depending only on $n$ and $s$.
    \end{proof}

{\begin{Remark}
    As a direct consequence of the proof of Proposition \ref{prop:s-poincare}, we obtain that for any measurable set $E \subset \mathbb{R}^n$ with $0 < |E| < \infty$, the following estimate holds \begin{equation}\label{controlloRieFra}
        \zeta_s(E)\geq C_0 \, \alpha^2(E),
    \end{equation}
   for some constant $C_0 > 0$ depending only on $n$ and $s$. 
{Observe moreover that, by \eqref{use-def-beta}, one trivially has that} \begin{equation*}
        \beta^2_s(E)\geq \zeta_s(E) \quad \text{ and }\quad \beta^2_s(E)\geq \delta_s(E).
    \end{equation*}
\end{Remark}}

	\section{Existence of generalized minimizers}\label{sec:exgen}
	
	As explained in the Introduction, in order to get our stability inequality, we find a constant $\varepsilon_0>0$, sufficiently small, so that problem \eqref{infclass} has the ball as the only solution. {Due to the competing nature of the terms $P_s$ and $V_s$, a direct proof of the existence of minimizers for \eqref{infclass} is not straightforward. For this reason, we consider a relaxed version of the functional, namely a generalized-type energy. In this section, we establish the existence of minimizers for this generalized functional. In the following section, we will then show that these minimizers are indeed minimizers for $\E_s$ in the classical sense.}
	
	\begin{Definition}
	We define the \textit{generalized functional} of $\E_s$ as the  functional
	\[
	\widetilde{\E_s} : (\R^n)^\mathbb N\to\R
	\]
	such that if $\widetilde E=\{E^k\}_{k\in\N}\in (\R^n)^\mathbb N$, then
\begin{equation}\label{defiGeneralziedFunctional}
	\widetilde{\E}_{s}\left(\{E^k\}_{k\in\N}\right) \coloneqq \sum_{k=1}^{\infty}\E_{s}(E^k).
\end{equation}
\end{Definition}
	The minimization problem we are interested in is the following:
	\begin{equation}\label{infgen}
		\min\left\{ \widetilde{\E}_{s}\left(\{E^k\}_{k\in\N}\right)\, :\, \sum_{k=1}^{\infty}|E^k| =\omega_n \right\}.
	\end{equation}
Let us begin by linking problem \eqref{infclass} and \eqref{infgen}. {The following result is closely related to \cite[Proposition 1.1]{CT}. Nevertheless, due to the presence of the nonlocal term $V_s$, the so-called \emph{vanishing range action} property (corresponding to assumption (H3) in \cite{CT}) is not satisfied in our setting. Despite this, the proof remains essentially the same, requiring only minor adjustments. Since the evidence of such adjustments is not immediate,  we include a detailed proof.}

	\begin{Proposition}\label{infclassinfgen}
		For any $\eee>0$, the equality $$\inf\left\{\mathcal{E}_s(E): |E|=\omega_n\right\}=\inf\left\{\widetilde{\E}_{s}\left(\{E^k\}_{k\in \N}\right): \sum_{k=1}^{\infty}|E^k|=\omega_n \right\}$$ holds.
	\end{Proposition}
\begin{proof} Let us set 
		\begin{equation*}
			I:=\inf\left\{\mathcal{E}_s(E): |E|=\omega_n\right\}
		\end{equation*}
		and 
		\begin{equation*}
			I_{gen}:=\inf\left\{\widetilde{\E}_{s}\left(\{E^k\}_{k\in \N}\right): \sum_{k=1}^{\infty}|E^k|=\omega_n \right\}.
		\end{equation*}
		
        Given a set $E$ with $|E|=\omega_n$, considering the sequence of sets $\gen{E}$ such that $E^1=E$ and $E^k=\emptyset$ for  $k\geq2$, we have 
		\[
		\widetilde{\mathcal{E}}_{s}(\gen{E}) = \mathcal{E}_s(E).
		\]
		Consequently, it follows that $I_{gen}\leq I$. 
		
		On the other hand, for any $\varepsilon > 0$ let  $\gen{E}$, with $\sum_{k}|E^k|=\omega_n$ be such that  
		\[
		\widetilde{\E}_s(\gen{E}) \leq I_{gen}+ \varepsilon.
		\]  
		We claim that there exists a set $E$ with $|E|=\omega_n$ such that 
		\[
		\mathcal{E}_s(E) \leq I_{gen}+ 3\varepsilon.
		\]  
	By (\emph{iii}) of Proposition \ref{elementary_properties_ps} and (\emph{iii}) of Proposition \ref{elementary_properties_vs}, for any ball $B_r$ centered at the origin of radius $r>0$, we have
		$$ \E_s(B_r)=r^{n-s}\E_s(B_1)\to 0\quad \text{as }\,r\to0^+.$$
		Thus, for any $\varepsilon>0$ there exists some $\eta= \eta(\varepsilon)$ such that
        \begin{equation}\label{stima_0}
         \text{if } \,|B_r| \leq 2\eta \,\text{ then }\, \mathcal{E}_s(B_r) \leq \varepsilon. 
        \end{equation} 
		Moreover, it is possible to choose $M = M(\eta)\in\N$ sufficiently large such that
		\begin{equation}\label{stima-1-equiv}
			\sum_{k=1}^M \mathcal{E}_s(E^k) \leq \sum_{k=1}^{\infty} \mathcal{E}_s(E^k)  \leq I_{gen}+ \varepsilon \quad \text{ and } \quad \sum_{k=M+1}^{\infty}|E^k| \leq  \eta.
		\end{equation}
		
		By (\emph{iv}) of Proposition \ref{elementary_properties_ps} and (\emph{iv}) of Proposition \ref{elementary_properties_vs}, since for any $E\subseteq \R^n$, $E\cap B_\rho\to E$ in $L^1$ as $\rho\to+\infty$, we get
		\begin{equation*}
			\E_s(E\cap B_\rho)\to \E_s(E), \quad \text{as }\,\rho\to+\infty.
		\end{equation*}
		Thus, there exists $R = R(\varepsilon, \eta, M)$ sufficiently large {(observe that all quantities $\eta,\;M,\;R$ just depends on $\varepsilon$)} such that  
		\begin{equation}\label{cutoff}  
			\sum_{k=1}^{M} \mathcal{E}_s(E^k \cap B_R) \leq \sum_{k=1}^M \mathcal{E}_s(E^k) + \varepsilon \leq I_{gen}+ 2\varepsilon \quad \text{ and } \quad \sum_{k=1}^{M} |E^k \cap B_{R}^c| \leq \eta.
		\end{equation}  
		Now, let us consider the ball $B_{r_1}$  centered at the origin such that 
		\[
		|B_{r_1}| = \sum_{k=1}^{M} |E^k \cap B_{R}^c| + \sum_{k=M+1}^\infty|E^k| \leq 2 \eta,
		\] 
       and, recalling \eqref{stima_0}, we have that
       \begin{equation}\label{stimaE(B_r1)}
           \E_s(B_{r_1})<\eee.
       \end{equation}
	Moreover, let 
		\[
		E_L := \bigg[\bigcup_{k=1}^{M} \big((E^k \cap B_{R}) + L^k e_1 \big)\bigg] \bigcup \bigg[ B_{r_1} +L^{M+1} e_1 \bigg],
		\]
		with $L > 0$, sufficiently large to ensure $|E_L| = \omega_n$. Then, for $L$ sufficiently large, using (\emph{i}) and (\emph{ii}) in Proposition \ref{elementary_properties_ps} and in Proposition \ref{elementary_properties_vs}, we have
		\begin{equation*}  
			\mathcal{E}_s(E_L) \leq \sum_{k=1}^M\mathcal{E}_s(E^k \cap B_R) + \mathcal{E}_s(B_{r_1}) \leq I_{gen}+ 3\varepsilon,
		\end{equation*}
		where the last inequality follows from \eqref{cutoff} and \eqref{stimaE(B_r1)}. Whence
		\[
		I \leq I_{gen}+ 3\varepsilon.
		\]  
		which concludes the proof by the arbitrariness of $\varepsilon > 0$.
	\end{proof}
	
For later purpose, with the scope of eliminating the volume constraint in Problem \eqref{infclass}, we introduce the following {{\it penalized} functional. Concretely, for any $\Lambda>0$, we define}
\begin{equation}\label{eq:EsL}
\begin{split}
	\EsL(\gen{E}):&=\widetilde{\E_s}(\gen{E})+\Lambda\left| \sum_{k=1}^\infty |E^k|-\omega_n\right|\\
    &=\sum_{k=1}^\infty \E_s(E^k)+\Lambda\left| \sum_{k=1}^\infty|E^k|-\omega_n\right|.
    \end{split}
\end{equation}

We have the following result.

\begin{Proposition}\label{prop:equivcunc}
{For any $\eee_0>0$ there exists a constant
{$\tilde\Lambda$ of the form $\tilde \Lambda=C_{n,s}(1+\eee_0)$ (for a suitable $C_{n,s}>0$)}
	  such that, for all $\eee<\eee_0$ and $\Lambda>\tilde\Lambda$, the following equality holds
	\begin{equation}\label{eq:equivinf}
\inf\left\{ \widetilde\E_{s}(\gen{E})\,:\,\sum_{k=1}^{\infty}|E^k|=\omega_n  \right\}=\inf_{E^k\subset\R^n}\EsL(\gen{E}) .
	\end{equation}
Moreover, any minimizer of the right-hand side of \eqref{eq:equivinf} satisfies the volume constraint $\sum_k |E^k|=\omega_n$, i.e., the constrained and the uncontrained generalized minimization problems for $\EsL$ are equivalent.}
\end{Proposition}
\begin{proof}
Let $\alpha:=\inf\left\{ \widetilde\E_{s}(\gen{E})\,:\,\sum_k|E^k|=\omega_n  \right\}{>0}$. {Given a sequence $\widetilde F=\gen{F}$ with $\sum_k|F^k|=\omega_n$, then 
	\[
		\widetilde{\mathcal{E}}_{s}(\gen{F}) = \EsL(\gen{F}).
	\]
      thus we have $\alpha\geq\inf_{E^k\subset\R^n}\EsL(\gen{E})$.

On the other hand, let us suppose that there exists $\widetilde E=\gen{E}$ such that
\begin{equation}\label{eq:contrL}
\EsL(\widetilde E)\leq\alpha.
\end{equation}
By testing with the sequence $\widetilde B=(B_1,\emptyset,\emptyset,\ldots)$ one gets 
\begin{equation}\label{stima_ast}
\Lambda\left|\sum_k |E^k|-\omega_n \right|\le \EsL(\gen{E})\leq \alpha\le \widetilde\E_s(\widetilde B)\le C_{n,s}(1+\varepsilon_0),
\end{equation}
where $C_{n,s}:=P_s(B_1)+\frac{n\omega_n}{n-s}$.

Now, let $\delta\in\R$ be such that $(1+\delta)^n\sum_k |E^k|=\omega_n$. Notice that this implies that $\delta\to0$ as $\Lambda\to+\infty$. Precisely, by \eqref{stima_ast}, one has
\begin{equation}\label{stima-ast2}
C\Lambda|\delta|\le\Lambda\omega_n\left|\frac{1}{(1+\delta)^n}-1 \right|\le C_{n,s}(1+\varepsilon_0),
\end{equation}
for some constant $C>0$, depending only on $n$, so that
\[
|\delta|\le \frac{C_{n,s}(1+\varepsilon_0)}{\Lambda},
\]
up to renaming the constant $C_{n,s}>0$.
Let now $\widetilde G=(1+\delta)\widetilde E=\{(1+\delta)E^k\}_{k\in\N}$. Then by \eqref{eq:contrL}, we have $\EsL(\widetilde E )\leq \alpha \leq \widetilde\E_s(\widetilde G)$, which implies, recalling the scaling property of $\widetilde\E_s$,
\[
\widetilde \E_{s}(\widetilde E)+\Lambda \left| \sum_k|E_k|-\omega_n \right|=\EsL(\widetilde E)\leq \widetilde\E_s(\widetilde G)\le  (1+\delta)^{n-s}\widetilde \E_{s}(\widetilde E)\le(1+2(n-s)|\delta|)\widetilde\E_{s}(\widetilde E),
\]
for $|\delta|$ small enough. Thus, by rearranging the terms, and taking into account \eqref{stima_ast} and \eqref{stima-ast2}, one gets
\begin{equation}\
\Lambda|\delta|\le 2(n-s)|\delta|\widetilde\E_{s}(\widetilde E)\le C_{n,s}(1+\eee_0)|\delta|.
\end{equation} 
{This implies that it must be $$\Lambda\le\widetilde\Lambda=C_{n,s}(1+\eee_0),$$ thus concluding
 the proof of the first claim.} By reasoning in the same way, the second follows similarly.}
\end{proof}

By the previous result, to show existence of constrained generalized minimizers for $\widetilde \E_{s}$ it is enough to show existence of unconstrained generalized minimizers for $\EsL$
\begin{equation}\label{eq:uncmin}
	\min_{E^k\subset \R^n} \EsL(\gen{E}).
\end{equation}
 { The proof of existence of minimizers for \eqref{eq:uncmin} follows a nowaday standard strategy, see, for example, \cite{CT, GNR4,  DiCNRV}.}
 \begin{Proposition}
 {Fixed $\eee_0>0$. Then, for any $\eee<\eee_0$ and $\Lambda>\tilde\Lambda$, with $\tilde \Lambda=C_{n,s}(1+\eee_0)$ given by Proposition \ref{prop:equivcunc}},  problem \eqref{eq:uncmin} (and thus problem \eqref{infgen}) admits a solution.
\end{Proposition}
\begin{proof}
 {The proof relies on the classical direct method in the Calculus of Variations. We begin by constructing a generalized minimizer using a standard minimizing sequence. We then prove that the associated functional is lower semicontinuous with respect to $L^1$-convergence, thereby ensuring that the limiting generalized set is indeed a minimizer for \eqref{eq:uncmin}.
 
 Let $\{E_h\}_{h\in\N}$ be a classical  minimizing sequence for the functional
    \begin{equation*}
        \E_{s,\Lambda}(E):=\E_s(E)+\Lambda \big| |E|-\omega_n\big|=P_s(E)+\eee V_s(E)+\Lambda \big||E|-\omega_n\big|.    \end{equation*} 
    By Proposition \ref{infclassinfgen}, $\{E_h\}_{h\ge 1}$ turns out to be a minimizing sequence even in the class of generalized sets. Using the ball $B_1$ as competitor, we have $$\sup_h \E_{s,\Lambda}(E_h)\le C_{n,s}(1+\varepsilon_0),$$
    where $C_{n,s}>0$ is the constant appearing in \eqref{stima_ast}. In particular, for $\Lambda>0$ sufficiently large, up to pass to a subsequence, $|E_h|\to m$, for some $ m\in (0,\infty)$. Fix $L:=10 m^{{1}/{n}}$ (so that in a cube of side lenght $L$, $E_h$ covers at most half of the volume of the cube, for $h$ large enough) and consider a partition of $\R^n$ into cubes $\{Q_{i}\}_{i\ge 1}$ with $Q_{i}:=[0,L]^n+z_i$, with $z_{i}\in (L\mathbb Z)^n$. Let $m_{i,h}:=|E_h\cap Q_{i}|$ be the mass of $E_h$ inside each cube. Let us rearrange the indexes so that, for every $h$, $m_{i,h}$ is nonincreasing in $i$.
	 
	We show that the sequence $\{m_{i,h}\}_{i\ge 1}$ is tight. With this purpose, let us consider the following localization of $P_s$ (see \cite[Subsection 1.2]{CT})
    \begin{equation}\label{Ps-loc}
        P_s(F;Q_i):=\int_{F\cap Q_i}\int_{\R^n\setminus F} \frac{dx\,dy}{|x-y|^{n+s}}
    \end{equation}
    for any measurable set $F\subseteq\R^n$.
    We first recall that
\begin{equation*}
P_s(E_h; Q_i) \geq \int_{E_h\cap Q_i}\int_{(\R^n\setminus E_h)\cap Q_i} \frac{dx\,dy}{|x-y|^{n+s}} = \frac{1}{2} \int_{Q_i}\int_{Q_i} \frac{|\chi_E(x) - \chi_E(y)|}{|x-y|^{n+s}}\,dx\,dy.
\end{equation*}
Then, since our choice of $L$ ensure that $|Q_{i}\cap E_h|\le |Q_{i}|/2$, by using the Poincaré-type inequality for fractional Sobolev spaces (see, e.g. \cite[Theorem 6.33]{Leo}) and proceeding exactly as in the proof of \cite[Lemma 2.5]{DiCNRV}, we obtain that there exists a constant $C = C(n,s,L)>0$ such that 
\[
C\,P_s(E_h; Q_i) \geq |E_h \cap Q_i|^{\frac{n-s}{n}}.
\]    
Hence, we have 
	\[
	\sum_{i} m^{\frac{n-s}{n}}_{i,h}\le C \sum_i P_s(E_h;Q_{i})= C\,P_s(E_h)\le c(n,s,L),
	\]
    for some constant $c(n,s,L)>0$.
	Since $m_{i,h}$ is nonincreasing in $i$, for every $I \geq 1$ and $i \geq I$ we have $m_{i,h}\le m_{I,h}\leq \frac{m}{I}$ which ensures
\begin{equation}\label{eq:tightper}
\sum_{i\ge I} m_{i,h}\le \left(\frac{m}{I}\right)^{\frac{s}{n}} \sum_{i\ge I} m_{i,h}^{\frac{n-s}{n}}\le  \left(\frac{m}{I}\right)^{\frac{s}{n}} c(n,s,L).
\end{equation}
    Hence, up to pass to a subsequence,  we have $\lim_{h\to \infty} m_{i,h}=:m_i$ and $\sum_i m_i=m$.

    From now on, we consider just the cubes $Q_i$ such that $m_{i,h}>0$. We rename such cubes $Q_{i,h}$, and coherently we relabel $z_{i,h}$ their centers. By the perimeter bound, up to extraction, we have for every $i$,  
	$E_h-z_{i,h}\to E^i$  in $L^1_{\rm loc}$ for some sets $E^i\subset\R^n$. We also freely assume that for every $i,j$, $|z_{i,h}-z_{j,h}|\to a_{ij}\in [0,\infty]$. 
	We define the following equivalence relation: we say that $i\sim j$ if $a_{ij}<\infty$ and denote by $[i]$ the equivalence class of $i$. 
	Note that if $i\sim j$ then  $E^i$ and $E^j$ are translated of each  other. For each class of equivalence we set 
	\[
	m_{[i]}:=\sum_{j\sim i} m_j 
	\]
	so that  $\sum_{[i]} m_{[i]}=m$. Then, for all $i$, using the local $L^1$-convergence of $E_h-z_{i,h}$ to $E^i$, we can deduce that $|E^i|=m_{[i]}$ (see, e.g., Step 1.3 in the proof of \cite[Theorem 1.2]{CT}).
	We now associate to each class of equivalence $[i]$ a single element $i$ (for instance, its first element) and set $\widetilde E=\{E^i\}_{i\in\N}$. We claim that the limiting generalized set constructed above is indeed a generalized minimizer, that is
    \begin{equation}\label{eq:toproveexistence}
	\sum_i\E_s(E^i)+ \Lambda \left|\sum_i |E^i|-\omega_n\right|\le \liminf_{h\to \infty} \big(\E_s(E_h)+\Lambda\big| |E_h|-\omega_n\big|\big).
	\end{equation} 
	
    Let us fix $I\in \N$ and $R>0$. Furthermore, for $h$ sufficiently large, we can assume that for $i,j\le I$ with $i\neq j$, $|z_{i,h}-z_{j,h}|\ge 5R$ which guarantees that $B_R(z_{i,h})\cap B_R(z_{j,h})=\emptyset$. We treat each term of the energy separately. Let us start with the $s$-perimeter, observing that
    \begin{equation*}
    \begin{split}
        \sum_{i=1}^I P_s(E_h-z_{i,h};B_R)&=\sum_{i=1}^I \int_{(E_h- z_{i,h})\cap B_R}\int_{\R^n\setminus (E_h-z_{i,h})}\frac{dx\,dy}{|x-y|^{n+s}}\\
        &=\sum_{i=1}^I \int_{E_h\cap B_R(z_{i,h})}\int_{\R^n\setminus E_h}\frac{dx\,dy}{|x-y|^{n+s}}\\
        &=\int_{E_h\cap \bigcup_{i=1}^I B_R(z_{i,h})}\int_{\R^n\setminus E_h}\frac{dx\,dy}{|x-y|^{n+s}}\\
        &\leq \int_{E_h}\int_{\R^n\setminus E_h}\frac{dx\,dy}{|x-y|^{n+s}}=P_s(E_h).      \end{split}
        \end{equation*}
    Then, recalling that $E_h-z_{i,h}\to E^i$ locally in $L^1$, by the lower semi-continuity of $P_s$ and by Fatou's Lemma, we obtain
    \begin{equation}\label{last-ps}
    \begin{split}
        \sum_{i=1}^I P_s(E^i)&\leq\sum_{i=1}^I \liminf_{h\to\infty}P_s(E_h-z_{i,h})\\
        &\leq \sum_{i=1}^I \liminf_{h\to\infty}\liminf_{R\to\infty}P_s(E_h-z_{i,h};B_R)\leq \liminf_{h\to\infty} P_s(E_h).   
        \end{split}    
    \end{equation}
    Similarly, concerning $V_s$, by (\emph{ii}) of Proposition \ref{elementary_properties_vs}, we have
    \begin{equation*}
        \sum_{i=1}^I V_s((E_h-z_{i,h})\cap B_R)=\sum_{i=1}^I V_s(E_h\cap B_r(z_{i,h}))\leq V_s\left(E_h\cap \bigcup_{i=1}^I B_R(z_{i,h})\right)\leq V_s(E_h),    
    \end{equation*}    
    which ensures, together with Fatou's Lemma,
    \begin{equation}\label{last-vs}
    \begin{split}
        \sum_{i=1}^I V_s(E^i)&\leq\sum_{i=1}^I \liminf_{h\to\infty}V_s(E_h-z_{i,h})\\
        &\leq \sum_{i=1}^I \liminf_{h\to\infty}\liminf_{R\to\infty}V_s((E_h-z_{i,h})\cap B_R)\leq \liminf_{h\to\infty} V_s(E_h).   
        \end{split}    
    \end{equation}
    Finally, for the penalization term, by tightness, it follows
\begin{equation}\label{last-volume}
\sum_{i=1}^\infty|E^i|=m=\lim_{h\to \infty} |E_h|.
\end{equation}
Hence, by \eqref{last-volume} and sending $I\to\infty$ in \eqref{last-ps} and \eqref{last-vs}, we obtain \eqref{eq:toproveexistence}, concluding the proof.}
	\end{proof}

	\section{Existence of minimizers}\label{sec:exmmin}
	In this section, we show that the generalized minimizers just found are classical minimizers for \eqref{infclass}, for $\varepsilon_0$ small enough. To this purpose, we recall the definition of $(\rho,R)$-minimizers of the $s$-perimeter, see \cite[Definition 1.3]{CT}. 
	
	\begin{Definition} \label{rho-min}
		Given a nonnegative and nondecreasing function $\rho : [0,+\infty [ \to \R$, we say that $E \subset \R^n$ is a $(\rho,R)$-minimizer of the $s$-perimeter if there exists $R>0$ such that for any $r \leq R$ and for any $x\in\mathbb{R}^n$, it holds
		\begin{equation}\label{eq-rho-min}
			P_s(E) \leq P_s(F) + \rho(r),
		\end{equation} 
		for any $F \subset \R^n$ with $E \Delta F\subset \subset B_r(x)$.
	\end{Definition}
	
	Now, we want to prove that any component of a generalized minimizer is a $(\rho,R)$-minimizer of the $s$-perimeter, with $\rho(r)=\eee Cr^{n-s}$, for some constant $C>0$ depending only on $n$ and $s$.
	
	\begin{Lemma}\label{Ek-rho-minim}
		Let $\eee_0>0$ be fixed and let $\widetilde{E}:=\gen{E}$ be a generalized minimizer of $\widetilde \E_s$ for $\eee\in (0,\eee_0]$. Then, for any $k\in \N$, the set $E^k$ is a $(\rho,R)$-minimizer for the $s$-perimeter with $\rho(r)=\eee_0\frac{10n\omega_n}{n-s} r^{n-s}$ and $R $ sufficiently small.
	\end{Lemma}
	
	\begin{proof} Let $\widetilde E=\gen{E}$ a minimizer for $\EsL$, with $\Lambda>\tilde \Lambda$, where $\tilde \Lambda=C_{n,s}(1+\eee_0)$ is given by Proposition \ref{prop:equivcunc}, (in light of Proposition \ref{prop:equivcunc}, $\widetilde E$ is a minimizer for $\widetilde\E_s$ as well). Let us fix $k\in\N$, $x\in \R^n$ and $R\leq1$ and let $F\subset\R^n$ be such that $E_k\Delta F\Subset B_r(x)$, for $r<R$. 

		Moreover, let us consider the sequence of sets $\widetilde{G}:=\left\{ G^l\right\}_{l\in\N}$ defined as $$G^l:=\left\{\begin{matrix}
			 F &\text{if}& l=k \\
			E^l &\text{if}& l\neq k.
		\end{matrix}\right.
		$$	
By minimality of $\widetilde{E}=\left\{E^l\right\}_{l\in\N}$ we have
		\begin{equation*}
		\begin{split}
			0&\leq \widetilde{\E}_s(\widetilde{G})-\widetilde{\E}_s(\widetilde{E})+\Lambda\left|\sum_l |G^l|-\omega_n\right|-\Lambda\left|\sum_l |E^l|-\omega_n\right|
			\\
			&=\sum_{\substack{l=1 \\ l\neq k}}^\infty(\E_s(G^l)-\E_s(E^l))+\E_s(G^k)-\E_s(E^k)+\Lambda\left|\sum_l |G^l|-\omega_n\right|-\Lambda\left|\sum_l |E^l|-\omega_n\right|\\
			&=\E_s(F)-\E_s(E^k) +\Lambda\left|\sum_l |G^l|-\omega_n\right|-\Lambda\left|\sum_l |E^l|-\omega_n\right|\\
			&\le \E_s(F)-\E_s(E^k) +\Lambda\big||E^k|-|F|\big|\\
			&\le \E_s(F)-\E_s(E^k) +\Lambda \omega_nr^n.
			\end{split}	
		\end{equation*}
		Hence, since $r<R\leq 1$, by rearranging the terms and using point (\emph{iv}) of Proposition \ref{elementary_properties_vs}, we get { 
		\begin{equation}\label{minimP_V}
			\begin{aligned}
             P_s(E^k)&\le P_s(F) + \varepsilon_0(V_s(F)-V_s(E^k))+\Lambda \omega_n r^n\\&\le  P_s(F)+\eee_0\frac{n\omega_n}{n-s} r^{n-s}+\Lambda \omega_n r^n \\&\leq P_s(F)+\eee_0\frac{10n\omega_n}{n-s}r^{n-s}.
            \end{aligned}
		\end{equation} 
	which concludes the proof.}
		
	\end{proof}
	
	As pointed out in \cite{GNR} when $s=1$, we cannot expect $C^{1,\alpha}$-regularity when $\rho$ has a critical power-like behavior. However, in light of \cite[Theorem 1.4, Remark 1.5 and Proposition 4.3]{CT}  when $\rho\sim \eee_0r^{n-s}$, with $\eee_0$ sufficiently small, we have that $(\rho,R)$-minimizers of the $s$-perimeter satisfy interior and exterior density estimates. 

{	In the following we will use the standard notation $E^{(t)}$ to denote the set of points of density $t$ of $E$, i.e. for $t\in [0,1]$, we set
	
 $$ E^{(t)}:=\left\{ x\in\R^n:\, \lim_{r\to 0^+}\frac{|E\cap B_r(x)|}{|B_r|}=t \right\}.$$}

We have the following result.

	\begin{Theorem}{\cite[Theorem 1.4]{CT}}\label{densityestimates}
	There exists $\overline{\varepsilon}\in(0,1)$ such that if  $E \subset \R^n $ is a $(\rho,R)$-minimizer of the $s$-perimeter with { $\rho(r)=\frac{10n\omega_n}{n-s}\,\overline{\varepsilon}\,r^{n-s}$}, then there exist {   constants} ${C_{DE}}>0$ and  ${r_{DE}}>0$ such that for any  $r \leq {r_{DE}}$, we have
		\begin{equation}\label{intdens}
			|E \cap B_r(x) | \geq {C_{DE}}\, r^n \quad\text{ for every } x \in E^{(1)}
		\end{equation}
		and
		\begin{equation}\label{outdens}
			|B_r(x) \setminus E| \geq {C_{DE}}\, r^n \quad \text{ for every } x \in E^{(0)},
		\end{equation}
	\end{Theorem}
	
	As a consequence of Lemma \ref{Ek-rho-minim} and Theorem \ref{densityestimates}, we get the following result.
	\begin{Corollary}\label{densityestimates-Ek}
		Let $\widetilde{E}:=\gen{E}$ be a generalized minimizer of $\widetilde \E_s$. Then there exist {   constants} ${C_{DE}}>0$ and  ${r_{DE}}>0$, such that for any  $r \leq {r_{DE}}$ and $k\in \N$, 
		\begin{equation}\label{intdens-Ek}
			|E^k \cap B_r(x) | \geq {C_{DE}}\, r^n \quad\text{ for every } x \in (E^k)^{(1)}
		\end{equation}
		and
		\begin{equation}\label{outdens-Ek}
			|B_r(x) \setminus E^k| \geq {C_{DE}}\, r^n \quad \text{ for every } x \in (E^k)^{(0)}.
	\end{equation}\end{Corollary}

    Using the density estimates for the components of generalized minimizers, we deduce the existence of a classical minimizer for \eqref{infclass} from the existence of a generalized one.	
	
    \begin{Theorem}[Existence of minimizers for \eqref{infclass}] \label{ex-cla}
		There exists a constant $\eee_0=\eee_0(n,s)>0$, such that, for every $\eee\in (0,\eee_0)$ the variational problem \eqref{infclass} admits a solution. 

	\end{Theorem}
	\begin{proof} Let $\widetilde{E}:=\{E^k\}_{k\in\N}$ be a minimizer for $\widetilde \E_{s}$. Since $\sum_{k=1}^{\infty}|E^k|=\omega_n$ and by \eqref{intdens-Ek}, we get
		\begin{equation}\label{use-den-est}
			|E^k|\geq |E^k\cap B_{r_{DE}}|\geq C_0\,r^n_{DE},
		\end{equation}
		for any $x\in (E^k)^{(1)}$, where $r_{DE}>0$ is the {   constant} given by Corollary \ref{densityestimates-Ek}, we conclude that $|E^k|\neq 0$ just for a finite number of indexes $k=1,\ldots,M$. We sort the indexes so that 
		$$ |E^1|\geq|E^2|\geq\ldots\geq |E^M|. $$So, in particular, $|E^{1}|\geq\frac{\omega_n}{M}$. Fixed $L>0$, let us consider the following family of sets 
		\begin{equation*}
			E^1_L:=E^1 \quad \text{and}\quad E^k_L:=E^k+L^ke_1, \quad \text{for }k=2,\ldots,M. 
		\end{equation*}
		Moreover, we define $F_L:=\bigcup_{k=1}^ME^k_L$. {By \eqref{use-den-est} and the volume constraint}, for $L > 0$ sufficiently large, we can assume that $|E_L| = \omega_n$ and  $|E^i_L\cap E^j_L|=0$, for every $i,j\in\{1,\ldots,M\}$, with $i\neq j$. By (\emph{i}) and (\emph{ii}) of Proposition \ref{elementary_properties_ps}, we have
		\begin{equation}\label{sud-add-FL}
			P_s(F_L)\leq \sum_{i=1}^M P_s(E^i_L)=\sum_{i=1}^M P_s(E^i).
		\end{equation}
		On the other hand, since $\widetilde{E}:=\gen{E}$, is a generalized minimizer and, as previously observed, $\widetilde{\E}_{s}(\widetilde{E})=\sum_{k=1}^M\E_{s}(E^k)$, by considering the sequence of sets $\widetilde{B}:=\gen{B}$, defined by
		$$ B^1:=B_1\qquad \text{and} \qquad B^k:=\emptyset \quad \text{for all }\,k\geq2, $$
		then
		\begin{equation*}
			0\geq\widetilde{\E}_{s}(\widetilde E)-\widetilde{\E}_{s}(\widetilde{B})=\sum_{k=1}^M\E_{s}(E^k)-\E_{s}(B_1)\end{equation*}
		which implies  
		\begin{equation}\label{min-gen-B1}
			\sum_{i=1}^M P_s(E^k)-P_s(B_1)\leq \eee \left(V_s(B_1)-\sum_{i=1}^M V_s(E^k)\right)\leq \eee \frac{n\omega_n}{n-s}.
		\end{equation}
		
		Furthermore, by \eqref{sud-add-FL} and recalling the quantitative estimate \eqref{isop-F2M3}, we have 
		\begin{equation}\label{quant-FL}
			\sum_{k=1}^M P_s(E^k)-P_s(B_1)\geq P_s(F_L)-P_s(B_1)\geq {  c_0(n,s)} \,|F_L\Delta B_1(x_0)|^2,
		\end{equation}
	    where the constant {$c_0(n,s):=P_s(B_1)\omega_n^{-1}\,c_Q(n,s)$ and $c_Q$ is the one appearing in \eqref{isop-F2M3}} and $x_0\in\R^n$ realizes the minimum in the Fraenkel asymmetry of $F_L$. By choosing $L>0$ sufficiently large, we can suppose there exists $k^\ast\in\{1,\ldots,M\}$ such that $|B_1(x_0)\cap E^{k^\ast}_L |>0$ and $|B_1(x_0)\cap E^{k}_L|=0$ for all $k\neq k^\ast$. Thus, in particular,
		\begin{equation*}
			\begin{split}
				|F_L\Delta B_1(x_0)| &\geq |E^{k^\ast}_L\Delta B_1(x_0)|+\sum_{\substack{k=1 \\ k\neq k^\ast}}^M|E^k_L|\\
				&=|E^{k^\ast}_L\Delta B_1(x_0)|+\sum_{\substack{k=1 \\ k\neq k^\ast}}^M|E^k|\geq |E^{k^\ast}_L\Delta B_1(x_0)|+(M-1){ C_{DE}\,r_{DE}^n},
			\end{split}
		\end{equation*}
		where in the last inequality we used \eqref{densityestimates-Ek}. This, together with \eqref{min-gen-B1} and \eqref{quant-FL}, ensures
		$$ c_0(n,s) \left(|E^{k^\ast}_L\Delta B_1(x_0)|+(M-1){ C_{DE}\,r_{DE}^n}\right) \leq \eee^{\frac 1 2} \left(\frac{n\omega_n}{n-s}\right)^{\frac 12},$$
		which leads to a contradiction, as soon as $\eee \in (0,\eee_0]$ with $\eee_0<{ c_0(n,s)}^2C_{DE}^2\,r^{2n}_{DE}\,\frac{n-s}{n\omega_n}$ unless $M=1$. This concludes the proof.
	\end{proof}
	
	\section{Regularity of minimizers}\label{sec:regularity} 
{In the previous section, we established the existence of minimizers for \eqref{infclass}. This result was obtained by relying on a mild regularity theory for $(\rho, R)$-minimizers of the $s$-perimeter, encoded in suitable density estimates. The goal of this section is to improve upon this regularity. First, we show that, for sufficiently small $\varepsilon$, minimizers of \eqref{infclass} are close in $L^1$-topology to a ball. Then, by further exploiting the density estimates, we upgrade this closeness to convergence in the Hausdorff distance. At this stage, we prove that minimizers of \eqref{infclass} are indeed $\Lambda$-minimizers of the $s$-perimeter in the sense of \cite[Section 3]{F2M3}, and thus we may apply the regularity results developed therein (specifically, \cite[Corollary 3.6]{F2M3}) to reduce to the case of nearly spherical sets.}

	\begin{Proposition} \label{L1-close} Let $\varepsilon_0=\eee_0(n,s)>0$ be given by Theorem \ref{ex-cla}, and  for $\eee\in (0,\eee_0]$, let $E_\eee$ be a minimizer of \eqref{infclass}.
		Then, as $\eee\to 0$, there exists $x_\varepsilon\in\R^n$ such that the sets $E_\eee-x_\varepsilon$ converge in  $L^1$-topology  to the unit ball $B_1=B_1(0)$. Namely there holds 
		\begin{equation}\label{stima-vic-L1}
			\left| (E_\eee-x_\varepsilon )\Delta B_1\right| \le { \gamma(n,s)}\eee^{1/2}.
		\end{equation}
	\end{Proposition}
	\begin{proof}
		By minimality of $ E_\eee$ in \eqref{infclass}, we get, 
		\begin{equation}\label{minim-Eeee}
			P_s(E_\eee)-P_s(B_1)\leq \eee (V_s(B_1)-V_s(E_\eee))\leq \eee V_s(B_1)=\eee \frac{n\omega_n}{n-s}.
		\end{equation}
		Then, recalling the quantitative estimate \eqref{isop-F2M3}, we have 
		\[
		\min_{x\in \R^n}\left| E_\eee \Delta B_1(x)\right|^2 \le 
		{ \gamma_0(n,s)}(P_s(E_\eee)-P_s(B)),
		\]
		with {$\gamma_0(n,s):=\frac{\omega_nc_Q}{P_s(B_1)}$ where the constant $c_Q$ is the one appearing in \eqref{isop-F2M3}.} This, together with \eqref{minim-Eeee}, gives \eqref{stima-vic-L1}, with {$\gamma(n,s):=(\gamma_0(n,s)\frac{n\omega_n}{n-s})^{1/2}$}.
	\end{proof}
	Next, we show, employing the density estimates found in Theorem \ref{densityestimates}, that the $L^1$-proximity can be improved into an $L^\infty$ (i.e. Hausdorff) one.
	
\begin{Proposition}\label{prop:Lambdareg}
Let $x_\varepsilon\in\R^n$ be the points found in Proposition \ref{L1-close}. Given $\delta>0$, there exists $\varepsilon(\delta)>0$, with $\varepsilon(\delta)\to0$ as $\delta\to0$, such that if $\varepsilon<\varepsilon(\delta)$ and $E_\varepsilon$ is a minimizer of \eqref{infclass}, then 
		\[
		\partial (E_\varepsilon-x_\varepsilon)\subset B_{1+\delta}(0)\setminus B_{1-\delta}(0).
		\]
        In particular, $E_\varepsilon-x_\varepsilon$ converges in Hausdorff metric to $B_1(0)$ as $\varepsilon\to0$.
	\end{Proposition}
	\begin{proof}
		{For simplicity of notation and without loss of generality, we may assume that $x_\varepsilon =0$.}
		By  Corollary \ref{densityestimates-Ek}, there exist two {constants
		$\bar\varepsilon$ and $r_{DE}$} such that \eqref{intdens} and \eqref{outdens} hold true for $E_\varepsilon$ as long as $\varepsilon(\delta)<\bar\varepsilon$.  {Let $x\in E_\varepsilon{ ^{(1)}}$ and suppose that $\dist( x, B_1)>\delta$. Then, if $r=r(\delta)=\frac{\delta}{2}$,  $B_{r}(x)\cap E_\varepsilon\subset E_\eee\cap B_1^c$. Hence, by \eqref{intdens} and \eqref{stima-vic-L1}, we obtain 
		\[
		r^{2n}\le { \frac{1}{C_{DE}^2}} |B_{r}(x)\cap E_\varepsilon|^2\le { \frac{1}{C_{DE}^2}} |E_\eee\Delta B_1|^2\le  { \frac{\gamma(n,s)^2}{C_{DE}^2}} \,\varepsilon  
		\]
		where $C_{DE}$ and $\gamma(n,s)$ are respectively the constants appearing in \eqref{intdens} and \eqref{stima-vic-L1}}. Now, if we select { $\varepsilon(\delta)=\frac{r^{2n}C_{DE}^2}{2\gamma^2({n,s})}$}, we get a contradiction for $\varepsilon<\varepsilon(\delta)$. A completely analogous argument shows that there can not be {$x\in B_1\cap E_\varepsilon^{(0)}$} with $\dist(x,\partial B_1)>\delta$, for $\varepsilon$ small enough. This leads to the desired Hausdorff convergence.

	\end{proof}
	By the Hausdorff proximity result, one gets immediately the following result.
	\begin{Corollary}\label{coroll-import}
		There exists $\bar\varepsilon>0$ such that for $\varepsilon<\bar\varepsilon$, if $E_\varepsilon$ is a solution of \eqref{infclass}, then there holds
		\[
		B_{1/2}(0)\subset E_\varepsilon. 
		\]
	\end{Corollary} 
    {At this stage, we have established that a minimizer satisfies a very weak notion of almost minimality. However, now that we know, up to translation, that its boundary lies away from the origin, which is the singularity point of the kernel defining $V_s$, we are in a position to upgrade this weak almost minimality to the stronger notion of $\Lambda$-minimality for the $s$-perimeter. To this end, we recall the definition of $\Lambda$-minimizers, see \cite[Section 3]{F2M3}.

    \begin{Definition}
        Let $\Lambda \ge 0$, $R > 0$, $s \in (0,1)$, and let $E \subset \mathbb{R}^n$ be a bounded measurable set. Then $E$ is said to be a $\Lambda$-minimizer of the $s$-perimeter if
\begin{equation}\label{eq:defLabda}
  P_s(E)\le P_s(F)+{\Lambda}|\,E\Delta F|\,,
\end{equation}
for every set $F\subset\R^n$, such that {  $F\Delta E\Subset B_r(x)$, with $r<R$ and $x\in\R^n$}.
    \end{Definition}

We now recall the following result from \cite[Corollary 3.6]{F2M3}, which plays a crucial role in our argument, as it allows us to reduce the analysis to the case of nearly spherical sets.

\begin{Proposition}\label{corollary3.6f2m3}
  Let $n\ge 2$, $\Lambda\ge0$, $s\in(0,1)$ and $E_h$, with $h\in\N$, be a $\Lambda$-minimizer of the $s$-perimeter. If $E_h$ converges in $L^1$ to $B_1$, then there exists a bounded sequence $\{\varphi_h\}_{h\in\N}\subset C^{1,\alpha}(\partial B_1)$, for some $\alpha\in(0,1)$ independent of $h$, such that
  $$
    \partial E_h=\Big\{(1+\varphi_h(x))x:x\in\partial B_1\Big\}\,\quad \text{ and }\quad \lim_{h\to\infty}\|\varphi_h\|_{C^1(\partial B_1)}=0.
$$   
\end{Proposition}

With Corollary \ref{coroll-import} at hand, we can now establish the main result of this section, namely the $\Lambda$-minimality of the solutions to \eqref{infclass}.}

	\begin{Proposition}\label{prove-lambda-min}
		There exists $\bar\varepsilon>0$ such that any minimizer $E_\varepsilon$ of \eqref{infclass}, with $\eee\in (0,\bar \eee]$, is a $\Lambda$-mimimizer for the $s$-perimeter, with $\Lambda=C(1+\overline{\varepsilon})$, for some constant $C>0$ depending only on $n$ and $s$.
		
	\end{Proposition}
	\begin{proof}
		{ Hereafter, in the present proof, for the sake of brevity and with no risk of confusion, we denote $E_\varepsilon$ simply by $E$.
		Let $x\in\R^n$ be fixed and $R>0$ small enough to be made precise later. Moreover, Let $F\subset\R^n$ be such that $F\Delta E\Subset B_r(x)$ with $r\leq R$ and $\lambda>0$ be such that $|\lambda F|=|E|=\omega_n$.} In particular, we get 
		\begin{equation}\label{stima-lamb}
		\begin{split}
			\lambda^{n-s}&=\left(\frac{|E|}{|F|}\right)^\frac{n-s}{n}=\left(\frac{\omega_n}{\omega_n+\big(|F|-|E|\big)}\right)^\frac{n-s}{n}\\
            &=\left(\frac{1}{1+\frac{|F|-|E|}{\omega_n}}\right)^\frac{n-s}{n}
            \le 1+c_{n,s}\big||E|-|F|\big|\le1+c_{n,s}|E\Delta F|.
		\end{split}
		\end{equation}
		which, together with the minimality of $E$, yields
		\begin{equation}\label{eq:stimareg}
			P_s(E)\le P_s(F)+c_{n,s} |E\Delta F|\,P_s(F)+c_{n,s}|E\Delta F|\,V_s(F) + \bar\varepsilon\left(\int_{F-{y_F}}\frac{dx}{|x|^s}-\int_{E-{y_E}}\frac{dx}{|x|^s}  \right),
\end{equation}
{where $y_F$ and $y_E$ denote, respectively, any $V_s$-center of $F$ and $E$. 

We now estimate each term on the right-hand side of \eqref{eq:stimareg} separately. Without loss of generality, we can suppose that \[
P_s(F)<P_s(E)< \E_s(E)\le \E_s(B_1)=C_{n,s}(1+\bar \eee),
\]
otherwise, if $P_s(E)\le P_s(F)$, there is nothing to prove. Moreover, by \eqref{eq:rearrineq}, for some constant $\tilde c_{n,s}>0$, {one has $V_s(F)\le V_s(F^*)\le \tilde c_{n,s}$.} Hence, up to renaming constants, \eqref{eq:stimareg} reads as
	\begin{equation}\label{final-lambda}
		P_s(E)\le P_s(F) +c_{n,s}(1+\bar\eee)|E\Delta F|+\bar\varepsilon\left(\int_{F-{y_F}}\frac{dx}{|x|^s}-\int_{E-y_E}\frac{dx}{|x|^s}  \right).
	\end{equation}
	It remains to estimate the last term. To this end, we observe that, upon choosing $\bar\varepsilon$ and $R$ sufficiently small, and by recalling Corollary \ref{coroll-import}, Proposition \ref{L1-close}, and Lemma \ref{lemma-center}, we may assume that
    $$B_{1/2}(0)\subset E \cap F, \quad |y_{E} | \leq \frac14 \quad \text{and} \quad |y_F | \leq \frac14,$$
    which ensures that    
		\[
		\begin{aligned}
			\int_{F-{y_F}}\frac{dx}{|x|^s}-\int_{E-{y_E}}\frac{dx}{|x|^s}&\le \int_{F-{y_F}}\frac{dx}{|x|^s}-\int_{E-{y_F}}\frac{dx}{|x|^s}\\
			&\le \int_{(F-{y_F})\setminus( E-{y_F})}\frac{dx}{|x|^s}-\int_{(E-{y_F})\setminus( F-{y_F})}\frac{dx}{|x|^s} \\
			&\le \int_{(F-{y_F})\Delta (E-{y_F})}\frac{dx}{|x|^s}\\
			&= \int_{\big((F-{y_F})\Delta (E-{y_F})\big)\setminus B_{1/2}(y_F)}\frac{dx}{|x|^s}\\
            &\le 4^s|(F-{y_F})\Delta (E-{y_F})|\\
            &=4^s|F\Delta E|.
		\end{aligned}
		\]
        This, together with \eqref{final-lambda}, concludes the proof upon choosing $\Lambda = C(1 + \overline{\varepsilon})$, for some constant $C > 0$ sufficiently large, depending only on $n$ and $s$.}
	\end{proof}
    
    {As a consequence of Proposition \ref{L1-close} and Proposition \ref{prove-lambda-min}, we get the following crucial regularity result.
	\begin{Corollary}\label{coro:ns}
		Let $\varepsilon_h\to0$ and let $E_h$ be minimizers for \eqref{infclass}, with $\eee=\eee_h$. Then, up to subsequence, $E_h$ are nearly spherical sets More precisely, there exists a sequence $\{\varphi_h\}_h\subset C^{1,\alpha}(\partial B_1)$ for some $\alpha\in(0,1)$ independent of $h$, such that
  $$
    \partial E_h=\Big\{(1+\varphi_h(x))x:x\in\partial B_1\Big\}\,\quad \text{ and }\quad \lim_{h\to\infty}\|\varphi_h\|_{C^1(\partial B_1)}=0.
$$ 
	\end{Corollary}}

	\section{Rigidity of the ball within the class of nearly spherical sets and proof of Theorem \ref{our-quant-beta}}\label{sec:rigidity}
	In this section, we show that in the class of nearly spherical sets, the ball is a rigid minimizer. 
    
	\begin{Theorem}\label{rigidity-min}
		There exists $\varepsilon_1>0$ such that if $\varepsilon<\varepsilon_1$ then the only minimizer for \eqref{infclass} is the ball $B_1$.
	\end{Theorem}
    
	\begin{proof}
		Up to select $\varepsilon_1>0$ small enough, by Corollary \ref{coro:ns}, we get that any minimizer of \eqref{infclass} is a nearly spherical set, whose boundary is parametrized by a function $\varphi_\varepsilon:\partial B_1\to\R$ such that $\|\varphi_\varepsilon\|_{C^1(\partial B_1)}\to0$, as $\eee\to 0$. Moreover, up to translations, we can suppose that the barycenters of the sets $E_\varepsilon$ are in the origin, that is
    $$\partial E_\varepsilon=\Big\{(1+\varphi_\varepsilon(x))x\,:\,x\in B_1\Big\},\quad |E_\eee|=\omega_n,\quad\text{and}\quad \int_{E_\eee} x\,dx=0, \quad \text{for }\eee\in(0,\eee_1].$$
		
    { 	Let $\eta_0\in(0,1/2)$ and $c_0>0$, the constants (depending only on $n$) given by \cite[Theorem 2.1]{F2M3}. Since $\|\varphi_\varepsilon\|_{C^1(\partial B_1)}\to0$, as $\eee\to 0$, up to choose $\eee_1$ smaller, we can assume that $\|\varphi_\varepsilon\|_{C^1(\partial B_1)}<\eta_0$ for $\eee\in(0,\eee_1]$. Thus, by \cite[Theorem 2.1]{F2M3}, we have
		\begin{equation}\label{eq:stanco}
		P_s(E_\varepsilon)-P_s(B_1)\ge c_0\,sP_s(B_1)\,\|\varphi_\varepsilon\|^2_{L^2(\partial B_1)}=:C_0(n,s) \|\varphi_\varepsilon\|^2_{L^2(\partial B_1)}. 
		\end{equation}}
        
		We claim that there exists $C_1=C_1(n,s)$ such that
		\begin{equation}\label{eq:contradictionfuglede}
			P_s(E_\varepsilon)-P_s(B_1)\le C_1 \varepsilon\|\varphi_\varepsilon\|^2_{L^2(\partial B_1)}.
		\end{equation}
		This, together with \eqref{eq:stanco},  immediately entails that $\varphi_\varepsilon=0$ for $\varepsilon$ small enough, so that $E_\varepsilon$ is a ball.
		
		To show \eqref{eq:contradictionfuglede}, we observe that by minimality of $E_\eee$, one has
		\[
		P_s(E_\varepsilon)-P_s(B_1)\le\varepsilon (V_s(B_1)-V_s(E_\varepsilon)).
		\]
Hence, the proof of \eqref{eq:contradictionfuglede} reduces to showing that there exists $C_1=C_1(n,s)>0$ such that
		
		\begin{equation}\label{eq:contradiction}
			V_s(B_1)-V_s(E_\varepsilon) \le C_1 \|\varphi_\varepsilon\|^2_{L^2(\partial B_1)}.
		\end{equation}
		In order to do that, we start by observing that, by definition of $V_s$, we have
		\begin{equation}\label{eq_V}
			V_s(B_1)-V_s(E_\varepsilon) \le \int_{B_1}\frac{1}{|x|^s}\,dx-\int_{E_\varepsilon}\frac{1}{|x|^s}\,dx=:I_0-I_1.
		\end{equation}
	By \eqref{Vs-balls}, we know that	\begin{equation}\label{I_0}
			I_0=\frac{n\omega_n}{n-s}.
		\end{equation}
        Moreover, we can compute $I_1$, using that $E_\varepsilon$ is a nearly spherical set parametrized by $\varphi_\varepsilon$, as follows
\begin{equation}\label{I_1}
			\begin{split}
				I_1=\int_E\frac{1}{|x|^s}\,dx&=\int_{\partial B_1}\int_0^{1+\varphi_\varepsilon(x)}\frac{\rho^{n-1}}{\rho^{s}}\,d\rho\,d\mathcal H^{n-1}(x)=
				\frac{1}{n-s}\int_{\partial B_1}(1+\varphi_\varepsilon(x))^{n-s}d\mathcal H^{n-1}(x).
			\end{split}
		\end{equation}
	By combining together \eqref{I_1} and \eqref{I_0}, we deduce
		$$I_0-I_1=\frac{n\omega_n}{n-s}-\frac{1}{n-s}\int_{\partial B_1}(1+\varphi_\varepsilon(x))^{n-s}d\mathcal H^{n-1}(x).$$
	
    Now, for $t \in (0,1)$, let us introduce the sets
		$$E_\varepsilon^t:=\{(1+t\varphi_\varepsilon(x))x\,:\,x \in \partial B_1\}$$
	and the function
		$$h(t):=\int_{\partial B_1}(1+t\varphi_\varepsilon(x))^{n-s}\,d\mathcal H^{n-1}(x).$$
Thus, in particular, we have
		\begin{equation*}
			E_\varepsilon^0=B_1,\quad E_\varepsilon^1=E,\quad h(0)=n\omega_n, \quad \text{and} \quad h(1)=\int_{\partial B_1}(1+\varphi_\varepsilon(x))^{n-s}d\mathcal H^{n-1}(x),
		\end{equation*}
		which ensures
		\begin{equation}\label{eq_I}I_0-I_1=\frac{1}{n-s}(h(0)-h(1)).
		\end{equation}
		
		In order to estimate the difference $h(0)-h(1)$, we argue similarly as in the proof of Theorem 2.1 in \cite{F2M3}.
		Since $|E|=\omega_n$, we have that
		\begin{equation}\label{h(0)}\int_{\partial B_1}(1+\varphi_\varepsilon(x))^n\,d\mathcal H^{n-1}(x)=n\int_{\partial B_1}\int_0^{1+\varphi_\varepsilon(x)}\rho^{n-1}\,d\rho\,d\mathcal H^{n-1}(x)=n|E|=n\omega_n=h(0).
		\end{equation}
		Hence 
		$$h(0)-h(1)=\int_{\partial B_1}(1+\varphi_\varepsilon(x))^n(1-(1+\varphi_\varepsilon(x)^{-s}))\,d\mathcal H^{n-1}(x).$$
By using a Taylor expansion, we have for any $|z|<1/2$:
		\[
		\begin{split}
			(1+z)^n\big(1-(1+z)^{-s}\big)&=\left(1+nz+\frac{n(n-1)}{2}z^2 +R_1(z)\right)\left(sz-\frac{s(s+1)}{2}z^2 + sR_2(z)\right)\\
			&\le sz +\left(sn-\frac{s(s+1)}{2}\right)z^2 +sR_3(z)
		\end{split},\]
		where $|R_3(z)|\le C(n)|z|^3$.
		Therefore, we get
		\begin{equation}\label{h}
			\begin{split}
				h(0)-h(1)&\le\int_{\partial B_1} \left(s\varphi_\varepsilon(x)+s\left(n-\frac{s+1}{2}\right)\varphi^2_\varepsilon(x)\right)\,d\mathcal H^{n-1}(x)+sC(n)\int_{\partial B_1}|\varphi_\varepsilon(x)|^3\,d\mathcal H^{n-1}(x)\\
                &\le \int_{\partial B_1} \left(s\varphi_\varepsilon(x)+s\left(n-\frac{s+1}{2}\right)\varphi^2_\varepsilon(x)\right)\,d\mathcal H^{n-1}(x)+sC(n)\int_{\partial B_1}\varphi^2_\varepsilon(x)\,d\mathcal H^{n-1}(x)			
			\end{split}
		\end{equation}
		where in the last inequality we have used that $\|\varphi_\varepsilon\|_{C^1(\partial B_1)}\le \eta_0<1/2$ to bound the $L^3(\partial B_1)$ norm of $\varphi_\varepsilon$ with its $L^2(\partial B_1)$ norm.
		
		On the other hand, the volume constraint \eqref{h(0)}, gives 
		$$\int_{\partial B_1} \left((1+\varphi_\varepsilon(x))^n-1\right)\,d\mathcal H^{n-1}(x)=0,$$
		which implies
		\begin{equation}
			\begin{split}
				\int_{\partial B_1}\varphi_\varepsilon(x)\,d\mathcal H^{n-1}(x)\le -\frac{n-1}{2}\int_{\partial B_1}\varphi_\varepsilon^2(x)\,d\mathcal H^{n-1}(x)+C(n)\int_{\partial B_1}\varphi_\varepsilon^2(x)\,d\mathcal H^{n-1}(x).
			\end{split}
		\end{equation}
		Inserting such estimate into \eqref{h}, we deduce
		
		\begin{equation}\label{final-h}
			\begin{split}
				h(0)-h(1)&\le -\frac{s(n-1)}{2}\|\varphi_\varepsilon\|^2_{L^2(\partial B_1)}+s\left(n-\frac{s+1}{2}\right)\|\varphi_\varepsilon\|^2_{L^2(\partial B_1)}+2sC(n)\|\varphi_\varepsilon\|^2_{L^2(\partial B_1)}\\
				&=\frac{s}{2}(n-s)\|\varphi_\varepsilon\|^2_{L^2(\partial B_1)}+2sC(n)\|\varphi_\varepsilon\|^2_{L^2(\partial B_1)}\le C_{n,s}\|\varphi_\varepsilon\|^2_{L^2(\partial B_1)}.
			\end{split}
		\end{equation}
		
		By combining \eqref{eq_V},\eqref{eq_I}, and \eqref{final-h}, we obtain
		\begin{equation}\label{final-vs}
		    V_s(B_1)-V_s(E_\varepsilon)\le I_0-I_1= \frac{1}{n-s}(h(0)-h(1))\le \tilde C_{n,s}\|\varphi_\varepsilon\|^2_{L^2(\partial B_1)},
            \end{equation}
		which concludes the proof of \eqref{eq:contradiction} and thus of the Theorem.
	\end{proof}

{ \begin{Remark}\label{rmk:fuglede}
For a nearly spherical set $E$, such that 
\[
\partial E = \left\{ (1 + \varphi(x))x \,:\, x \in \partial B_1 \right\}, \quad |E| = \omega_n, \quad \int_E x \, dx = 0, \quad \text{and} \quad \|\varphi\|_{C^1(\partial B_1)} < \eta_0,
\]
where $\eta_0 = \eta_0(n) \in (0,1/2]$ is the constant from \cite[Theorem 2.1]{F2M3}, by \cite[Theorem 6.2]{DiCNRV}, there exist two constants $c_1, c_2 > 0$, depending only on $n$ and with $c_1 < c_2$, such that
\[
c_1 \left( [u]^2_{H^{\frac{1+s}{2}}(\partial B_1)} + s P_s(B_1)\|u\|^2_{L^2(\partial B_1)} \right) \le P_s(E) - P_s(B_1) \le c_2 [u]^2_{H^{\frac{1+s}{2}}(\partial B_1)}.
\]
Hence, recalling \eqref{use-def-beta}, we obtain
\[
\beta_s^2(E) \ge \delta_s(E) \ge \frac{c_1}{P_s(B_1)} [u]^2_{H^{\frac{1+s}{2}}(\partial B_1)} + s c_1 \|u\|^2_{L^2(\partial B_1)}.
\]
Moreover, by \eqref{final-vs},
\[
\beta_s^2(E) = \delta_s(E) + \zeta_s(E) \le \frac{c_2}{P_s(B_1)
} [u]^2_{H^{\frac{1+s}{2}}(\partial B_1)} + \frac{c_{n,s}\tilde C_{n,s}}{n \omega_n P_s(B_1)} \|u\|^2_{L^2(\partial B_1)},
\]
where $\tilde C_{n,s} > 0$ is the constant appearing in \eqref{final-vs}. Therefore, there exist constants $C_1, C_2 > 0$, depending only on $n$ and $s$, with $C_1 < C_2$, such that
\[
C_1 \left( [u]^2_{H^{\frac{1+s}{2}}(\partial B_1)} + \|u\|^2_{L^2(\partial B_1)} \right) \le \beta_s^2(E) \le C_2 \left( [u]^2_{H^{\frac{1+s}{2}}(\partial B_1)} + \|u\|^2_{L^2(\partial B_1)} \right).
\]
{In particular, making use of the Poincar\'e-type inequality \eqref{eq:s-poincare}, we deduce that our quantitative estimate \eqref{eq:our-quant-beta} is essentially equivalent to the Fuglede-type estimate \eqref{fugF2M3} for nearly spherical sets.}
\end{Remark}
}
\begin{customproof}{Theorem \ref{our-quant-beta}}	{	In view of Proposition \ref{prop:s-poincare}, it is enough to show that there exists a constant $C$ such that 
		$$\beta_s^2(E)\le C\delta_s(E).$$
		
		Since the quantities $\beta_s(E)$ and $\delta_s(E)$   are scale invariant, we shall assume from now on and without loss of generality that $|E|=\omega_n$ and $B_r=B_1$. By Theorem \ref{rigidity-min}, we have
		\begin{equation*}
			P_s(E)+\eee_1V_s(E)\geq P_s(B_1)+\eee_1V_s(B_1),
		\end{equation*}
		which implies 
		\begin{equation}\label{eq:last-est}
        \delta_s(E)\geq C\,\zeta_s(E),	\end{equation}
        for some constant $C>0$ depending only on $n$ and $s$. Since by \eqref{use-def-beta} we have that $\beta_s^2(E)=\delta_s(E)+\zeta_s(E)$, this concludes the proof of the Theorem.}
        \end{customproof}

{ 
\begin{Remark}[Asymptotic behavior as $s\to 1^-$]\label{rmk-asy}
Once the stability estimate \eqref{our-quant-beta} has been established, we now turn to its asymptotic behaviour as $s \to 1^-$. In particular, our goal is to verify that, in the limit $s \to 1^-$, one recovers the classical stability estimate \eqref{mainFJ} for the perimeter, proved in \cite{FJ}. 

Thanks to the well-known asymptotic properties of $P_s$, it follows that $\beta_s(\cdot)\to \beta(\cdot)$ and hence $A_s(\cdot)\to A(\cdot)$ as $s \to 1^-$. Therefore, the crucial point is to keep precise track of the dependence on $s$ of the constant $c(n,s)>0$ appearing in \eqref{eq:our-quant-beta}, and thus of all the constants involved in our argument. It is straightforward to recognize that, throughout Sections \ref{sec:prel}, \ref{sec:exgen}, \ref{sec:exmmin}, \ref{sec:regularity} and \ref{sec:rigidity}, all the constants depending on $s$ are explicit, except for $c_Q$ in \eqref{isop-F2M3} and the constants $\overline{\varepsilon}$, ${C_{DE}}$ and ${r_{DE}}$ involved in the density estimate of Theorem \ref{densityestimates}.

The constant $c_Q$, appearing in the quantitative fractional isoperimetric inequality \eqref{isop-F2M3}, is uniformly bounded as $s\to 1^-$, as observed in \cite[Remark 1.2]{F2M3}.  

Let us now focus on the constants in Theorem \ref{densityestimates}. By the proof of \cite[Theorem 1.4]{CT}, one sees that the dependence on $s$ of these constants is determined by the structural parameters in assumption (S3) of \cite{CT}, which are necessary for the validity of density estimates. For this reason, we briefly check them here. This analysis was partly carried out in \cite[Proposition 4.3]{CT}, but we provide the details to make the dependence on $s$ completely explicit. 

First, considering the localization of $P_s$ in \eqref{Ps-loc}, by \cite[Proposition 2.5]{DiCNRV} and \cite[Theorem 1]{BBM}, assumption (H4) of \cite{CT} (relative isoperimetric inequality) is fulfilled with
\begin{equation}\label{def_f1}
f_1(m):=C_1(n,s)m^{\frac{n-s}{n}}, \qquad \text{for }m>0,
\end{equation}
for any $r_0>0$, where $C_1(n,s)\sim \tfrac{1}{1-s}$ as $s\to 1^-$. On the other hand, assumptions (H6) and (H12) with $C=1$ follow directly from \eqref{Ps-loc}.

Regarding (H13), let $E\subset \R^n$ and $x \in \R^n$. By translation invariance, we may assume $x=0$. Fixing $t>0$, for any $z \in B_t$, we compute
\[
\int_{B_t^c} \frac{dy}{|z-y|^{n+s}} = \int_{B_t(-z)^c} \frac{dy}{|y|^{n+s}} \leq \int_{B_{t-|z|}^c} \frac{dy}{|y|^{n+s}} = \frac{n\omega_n}{s}\frac{1}{(t-|z|)^s},
\]
which, together with the coarea formula, yields for a.e.\ $t>0$:
\begin{equation*}
 P_s(B_t; E) = \int_{E \cap B_t} \int_{B_t^c}\frac{dzdy}{|z-y|^{n+s}} \leq \frac{n\omega_n}{s}\int_{E \cap B_t} \frac{dz}{(t-|z|)^{s}}
= \frac{n\omega_n}{s}\int_{0}^t \frac{\mathcal{H}^{n-1}(E \cap \partial B_{\rho})}{(t-\rho)^s}\, d\rho.
\end{equation*}
Hence, by applying the Fubini-Tonelli Theorem, for any $r>0$, we obtain \[ \begin{split} \frac{1}{r}\int_0^r& P_s(B_t; E) \, dt \\&=\frac{n\omega_n}{sr}\int_0^r\int_0^t \frac{\mathcal{H}^{n-1}(E \cap \partial B_{\rho})}{(t-\rho)^s} \, d\rho \,dt =\frac{n\omega_n}{sr}\int_0^r \int_\rho^r \frac{\mathcal{H}^{n-1}(E \cap \partial B_{\rho})}{(t-\rho)^s} \, dt\, d\rho \\&= \frac{n\omega_n}{sr}\frac{(r-\rho)^{1-s}}{1-s}\int_0^r\mathcal{H}^{n-1}(E \cap \partial B_{\rho})\,d\rho \leq \frac{n\omega_n}{sr}\frac{r^{1-s}}{1-s}\, |E \cap B_r| = \frac{n\omega_n}{s(1-s)}\, r^{-s}\, |E \cap B_r|. \end{split} \]
Thus, condition (H12) in \cite{CT} holds with
\begin{equation}\label{def_f2}
f_2(r,m) := C_2(n,s)\, r^{-s} m, \qquad \text{for }m>0 \ \text{ and for any  }\ r>0,
\end{equation} 
where $C_2(n,s):= \tfrac{n\omega_n}{s(1-s)}$.

Finally, let us verify (H14). By Lemma \ref{Ek-rho-minim}, $\rho(r)=C_0(n,s)\bar \varepsilon r^{n-s}$, with $C_0(n,s):=\tfrac{10n\omega_n}{n-s}$. Combining this with \eqref{def_f1} and \eqref{def_f2}, we obtain
\[
f_3(r,m) := 2^n\left(\frac{f_2(r, m) + \rho(r)}{f_1(m)}\right) =\frac{C_2}{2^{-n}C_1}\Big(\frac{m}{r^n}\Big)^{\frac{s}{n}} + \frac{C_0\bar \varepsilon}{2^{-n}C_1}\Big(\frac{m}{r^n}\Big)^{\frac{s-n}{n}}.
\]
Setting $\delta := \tfrac{m}{r^n}$, one checks that 
\[
f_3(r,m)\leq 1 \quad \text{whenever} \quad \delta \in \left(\tfrac{\eee_1}{2^n},\eee_1\right],
\]
for some suitable $\eee_1>0$ and $r_3>0$. More precisely, $f_3(r,m)\leq 1$ provided that $\eee_1 \in \mathcal{I}:=(a,b]$, with
\[
a(n,s):=2^n\Big(\frac{2^{n+1} C_0\bar \varepsilon}{C_1}\Big)^{\tfrac{n}{n-s}} \sim (1-s)^{\tfrac{n}{n-s}}, \qquad 
b(n,s):=\Big(\frac{C_1}{2^{n+1}C_2} \Big)^{\tfrac{n}{s}} \sim s^{\tfrac{n}{s}} \quad \text{as } s \to 1^-.
\]
According to \cite[Remark 1.5]{CT}, the interval $\mathcal{I}$ is well defined as soon as $\bar \eee$ is sufficiently small. Consequently, all the parameters involved in assumption (H14), as well as $\bar\eee$, ${C_{DE}}$ and ${r_{DE}}$ in Theorem \ref{densityestimates}, remain uniformly bounded as $s\to 1^-$.  

In conclusion, as $s\to 1^-$, the estimate \eqref{eq:our-quant-beta} recovers the stability result for the classical perimeter in \eqref{mainFJ}. 
\end{Remark}}

	\section{A quantitative fractional Cheeger inequality}\label{sec:cheeger}
	In this Section, we show some quantitative versions of the fractional Cheeger inequality, which is the nonlocal analogue of the result contained in \cite{FMP} and \cite{JS}.
	
	Let $\Omega \subset \mathbb R^n$ be an open set with finite measure and let $m>\frac{n-s}{n}$. We recall that the \textit{fractional Cheeger constant} (or \textit{$(m,s)$-Cheeger constant}) of $\Omega$ is given by
	$$
	h_{m,s}(\Omega):=\inf\left\{\frac{P_s(E)}{|E|^m},\;\;E\subset \Omega\right\}.$$
	We call a \textit{fractional Cheeger set} or a \textit{$(m,s)$-Cheeger set} any minimizer of the above infimum problem, whose existence follows by a standard application of the classical direct method in the Calculus of Variations (see e.g. Proposition 5.3 in \cite{BLP}, for the case $m=1$).
	
	Let $B$ denote the Euclidean unit ball. The following \textit{fractional Cheeger inequality} holds true
	\begin{equation}\label{fractional-cheeger}|\Omega|^{m-\frac{n-s}{n}} h_{m,s}(\Omega)\ge |B|^{m-\frac{n-s}{n}} h_{m,s}(B).\end{equation}
	
	For the case $m=1$, the fractional Cheeger constant and properties of fractional Cheeger sets have been studied in \cite{BLP}.
      
    Also, we say that $\Omega$ is {\it $(m,s)$-calibrable} if it is an $(m,s)$-Cheeger set of itself, i.e. if
\[
h_{m,s}(\Omega)=\frac{P_s(\Omega)}{|\Omega|^m}.
\]
\begin{Remark}\label{C_ball}   
As in the local case, any ball $B\subset\mathbb{R}^n$ is $(m,s)$-calibrable. This is a direct consequence of the nonlocal isoperimetric inequality \eqref{non-isop-in}, which gives for every $E\subset B$
\begin{equation*}
\frac{P_s(E)}{|E|^m}\ge \frac{P_s(B)}{|B|^m}\, \left(\frac{|B|}{|E|}\right)^\frac{sm}{n}\ge \frac{P_s(B)}{|B|^m}.
\end{equation*}
\end{Remark}

The aim of this Section is to provide a quantitative version of \eqref{fractional-cheeger}, following the approach in \cite{JS}.

    \begin{customproof}{Theorem \ref{Thm:Ch-2}} 
    Thanks to the scaling property of the Cheeger constant, i.e. $h_{m,s}(\lambda \Omega) = \lambda^{(1-m)n-s}h_{m,s}(\Omega)$ for any $\lambda>0$, it is enough to prove the claim for $\Omega$ such that $|\Omega|=\omega_n$ and $B_r=B$ is the unit ball.
		
		Moreover, since by \eqref{Vs-balls}, we know that $V_s(B)=\frac{n\omega_n}{n-s}$, it follows that
        { $\zeta_s(\Omega)\le 1$}. Therefore, the inequality \eqref{quant-cheeger-2} immediately follows for sets $\Omega$ such that  $h_{m,s}(\Omega) \ge 2h_{m,s}(B)$, by choosing { $\kappa_{m,n,s}\leq 1$}. Therefore, hereafter, let us consider $\Omega$ with volume $\omega_n$ and such that $h_{m,s}(\Omega) < 2h_{m,s}(B)$. Let us denote by $E$ a $(m,s)$-Cheeger set of $\Omega$.
		
		First of all, let us  prove the two following estimates
		\begin{equation}\label{first}
			|E|\ge |\Omega|\left( \frac{h_{m,s}(B)}{h_{m,s}(\Omega)} \right)^\frac{1}{m-\frac{n-s}{n}}
		\end{equation}
		and \begin{equation}\label{second}
			\delta_s(E) \le \frac{h_{m,s}(\Omega)-h_{m,s}(B)}{h_{m,s}(B)}.
		\end{equation}
		
		We start by proving \eqref{first}. By the fractional isoperimetric inequality \eqref{non-isop-in}, we know that
		$$\frac{P_s(E)}{|E|^m}\ge \frac{P_s(B)}{|B|^\frac{n-s}{n}}\cdot|E|^{\frac{n-s}{n}-m}=:\mathcal I_{n,s}|E|^{\frac{n-s}{n}-m}$$
        where $\mathcal I_{n,s}:=\frac{P_s(B)}{|B|^\frac{n-s}{s}}$ denotes the $s$-isoperimetric constant. Hence, since $E$ is a $(m,s)$-Cheeger set of $\Omega$, by Remark \ref{C_ball}, we deduce that
		$$|E|^{m-\frac{n-s}{n}}\ge \frac{P_s(B)}{|B|^{\frac{n-s}{n}}}\cdot \frac{1}{h_{m,s}(\Omega)}=|B|^{m-\frac{n-s}n}\frac{h_{m,s}(B)}{h_{m,s}(\Omega)},$$
		which proves \eqref{first}.
		
		Let us show now \eqref{second}. Again by Remark \ref{C_ball}, we have that
		$$h_{m,s}(\Omega)-h_{m,s}(B)\ge \frac{P_s(E)}{|E|^{\frac{n-s}{n}}}\cdot |E|^{\frac{n-s}{n}-m} - \frac{P_s(B)}{|B|^{\frac{n-s}{n}}}\cdot |B|^{\frac{n-s}{n}-m}.$$
        Diving the above inequality by $\mathcal I_{n,s}$ and recalling the definition of $\delta_s(E)$, we get
		\[
		\begin{split}\frac{h_{m,s}(\Omega)-h_{m,s}(B)}{\mathcal I_{n,s}}&\ge (1+\delta_s(E))|E|^{\frac{n-s}{n}-m}-|B|^{\frac{n-s}{n}-m}\\
			&= \delta_s(E)|E|^{\frac{n-s}{n}-m} + \left( |E|^{\frac{n-s}{n}-m}-|B|^{\frac{n-s}{n}-m}\right).
		\end{split}
		\]
		Since $m> \frac{n-s}{n}$ and $|E| \le |\Omega|=|B|$, we have that the last term on the right-hand side is non-negative. This implies that
		$$
		\frac{h_{m,s}(\Omega)-h_{m,s}(B)}{\mathcal I_{n,s}}\ge\delta_s(E)|E|^{\frac{n-s}{n}-m}\ge \delta_s(E)|B|^{\frac{n-s}{n}-m}.
		$$
		Recalling the definition of $\mathcal I_{n,s}$ and using again Remark \ref{C_ball}, we obtain
		$$ \delta_s(E) \le \frac{h_{m,s}(\Omega)-h_{m,s}(B)}{h_{m,s}(B)}, $$
		and thus the proof of \eqref{second} is concluded.        
        
        Now, let us start by estimating $\zeta_s(\Omega)$ in terms of $\zeta_s(E)$. Up to translating $E$ (and therefore $\Omega$) we  may  assume that $E$ is $V_s$-centered at the origin, i.e.
		\begin{equation*}
            { V_s(E)=\int_{E}\frac{1}{|x|^s}\,dx}       
        \end{equation*}
		By setting $r_E>0$ the radius of the ball $E^*$ and adding and removing {$r_{E}^{n-s}$} in the definition of $\zeta_s(\Omega)$
		we obtain
        { 
		\begin{equation}\label{eq:eq1}
        \begin{split}
			  \zeta_s(\Omega) 
            &= 1-r_E^{n-s}+r_E^{n-s}-\frac{V_s(\Omega)}{V_s(B)}
            =1-\frac{P_s(E^*)}{P_s(B)}+\frac{V_s(E^*)}{V_s(B)}-\frac{V_s(\Omega)}{V_s(B)} \\
			&\leq\,\frac{P_s(B)-P_s(E^*)}{P_s(B)} +\frac{V_s(E^*)-V_s(E)}{V_s(E^*)}
            =\,\frac{P_s(B)-P_s(E^*)}{P_s(B)}  +\zeta_s(E).
            \end{split}
        \end{equation}}
		where in the first equality we have used ($iii$) of Proposition \ref{elementary_properties_ps} and ($iii$) of Proposition \ref{elementary_properties_vs} while the last inequality follows by observing that since $E\subset \Omega$, we have $E^*\subset B$, $V_s(E)\leq V_s(\Omega)$ and {$V_s(E^*)\leq V_s(B)$}.
		
		Now to get \eqref{quant-cheeger-2}, we want to estimate both terms on the right hand side through the re-normalized difference of the Cheeger constants $(h_{m,s}(\Omega)-h_{m,s}(B))h_{m,s}(B)^{-1}$. At first, by exploiting \eqref{first}, we get 
		\begin{equation*}
			\frac{P_s(B)-P_s(E^*)}{P_s(B)} 
			=
			1 - \left( \frac{|E|}{|\Omega|}\right)^{\frac{n-s}{n}} \leq 1-\left( \frac{h_{m,s}(B)}{h_{m,s}(\Omega)}\right)^{\frac{n-s}{mn-(n-s)}}
		\end{equation*}
		consequently, since $h_{m,s}(\Omega)\geq h_{m,s}(B)$, it follows
		\begin{equation*}
			\frac{P_s(B)-P_s(E^*)}{P_s(B)} 
			\leq 1-\left( \frac{h_{m,s}(B)}{h_{m,s}(\Omega)}\right)^a
			=\frac{h_{m,s}(\Omega)^a-h_{m,s}(B)^a}{h_{m,s}(\Omega)^a}
		\end{equation*}
		with $a:=\frac{n-s}{mn-(n-s)}+1\geq1$. Then, by recalling the standard inequality $t^a -\tau^a \le a t^{a-1}(t-\tau)$ whenever $a\ge 1$, and $\tau\in (0,t]$, we finally obtain
		\begin{equation}\label{eq:eq2}
			\frac{P_s(B)-P_s(E^*)}{{P_s(B)}} \le \left(\frac{n-s}{mn-(n-s)}+1\right)\frac{h_{m,s}(\Omega) - h_{m,s}(B)}{h_{m,s}(\Omega)}.
		\end{equation}
		
		Finally, in order to estimate $\zeta_s(E)$, by \eqref{second} and \eqref{eq:last-est} we have 
		\begin{equation}\label{stima-zetaE}
			\zeta_s(E)\leq \frac{1}{C}\delta_s(E)\leq \frac{1}{C}\frac{h_{m,s}(\Omega) - h_{m,s}(B)}{h_{m,s}(\Omega)},
		\end{equation}
		where the constant $C$ is the one appearing in \eqref{eq:last-est}. Thus, combining \eqref{eq:eq1}, \eqref{eq:eq2} and \eqref{stima-zetaE}, we find
        { 
		\begin{equation*}
				\zeta_s(\Omega) \leq \left(\frac{n-s}{mn-(n-s)}+1\right)\left(\frac{h_{m,s}(\Omega) - h_{m,s}(B)}{h_{m,s}(\Omega)}\right)
                +\frac{1}{C} \left(\frac{h_{m,s}(\Omega) - h_{m,s}(B)}{h_{m,s}(\Omega)}\right)
		\end{equation*}}
		which yields the claim by setting { $$\kappa({m,n,s}):=\left(\frac{n-s}{mn-(n-s)}+1+\frac{1}{C}\right)^{-1}.$$}
	\end{customproof}

In light of \eqref{controlloRieFra}, a quantitative version of \eqref{fractional-cheeger}, involving Frankel asymmetry holds.
\begin{Corollary}\label{Thm:Ch-1}
		Let $\Omega \subset \mathbb R^n$ be an open set with finite measure and let $m>\frac{n-s}{n}.$ Then, there exists a constant $\gamma(m,n,s)>0$, such that
		\begin{equation}\label{quant-cheeger}
            \frac{h_{m,s}(\Omega) - h_{m,s}(B_r)}{h_{m,s}(B_r)} \ge  \gamma(m,n,s)\,\alpha^2(\Omega),\qquad \text{with } |B_r|=|\Omega|.            
		\end{equation}
\end{Corollary}

We also point out that, once \eqref{first} and \eqref{second} are established, inequality \eqref{quant-cheeger} can be readily deduced by suitably generalizing the approach developed in \cite{FMP}.

  {\begin{Remark}
According to Remark~\ref{rmk-asy}, in the limit $s\to 1^-$, the estimates \eqref{quant-cheeger-2} and \eqref{quant-cheeger} reduce to the classical stability results for the Cheeger constant obtained in~\cite[Theorem~2.1]{JS} (for $m=1$) and~\cite{FMP}, respectively.
    \end{Remark}}
    
\subsection{Failure of the stability estimate for fractional Cheeger constant with the index \texorpdfstring{$\beta_s$}{betas}.}\label{ssec:cheeger-beta} 
{At this stage, it is natural to investigate whether a stability estimate for the fractional Cheeger constant involving the oscillation index $\beta_s$ could hold. In the local setting, the failure of such an estimate was demonstrated in \cite[Subsection 2.2]{JS}. First, we show that one of the counterexamples constructed in~\cite{JS}, based on planar sets with increasing boundary oscillations, also applies in the nonlocal setting. We then present an alternative counterexample inspired by the fractal-type sets discussed in~\cite{Lom}, which turn out to be suitable in the nonlocal case as well.

More precisely, we show that there does not exist any constant $c>0$ such that the inequality
\begin{equation}\label{eq:failure_beta_s}
\frac{h_{m,s}(\Omega) - h_{m,s}(B_r)}{h_{m,s}(B_r)} \ge c\, \beta_s^2(\Omega),\qquad \text{with } |B_r|=|\Omega|,
\end{equation}
holds. This is proved by constructing some suitable families of sets.

First of all, fix a universal constant $\varepsilon>0$ sufficiently small, and consider the family of bounded planar sets $\{\Omega^\varepsilon_j\}_{j\in \mathbb{N}}$, whose boundary is given by
\begin{equation*}
\partial \Omega^\varepsilon_j
= \{\, \theta(1+\eee\,u(\theta))\,:\, \theta\in \partial B_1\},  \qquad \text{with } \ u(\theta):=\sin(2j\,\theta),
\end{equation*}
where we systematically identify $\partial B_1\equiv \mathbb{S}^1$ with the interval $[0,2\pi]$.
The volume of $\Omega^\varepsilon_j$ is
\begin{equation}\label{calcolovolume}
\begin{split}
    |\Omega^\varepsilon_j|
    &=\frac{1}{2} \int_{0}^{2\pi} \bigl( 1 + u(\theta) \bigr)^{2} \, d\theta
    = \frac{1}{2} \int_{0}^{2\pi} \bigl( 1 + 2u(\theta) + u^{2}(\theta) \bigr) \, d\theta\\
    &= \pi+\int_0^{2\pi}\sin^2(2j\,\theta)\,d\theta
    = \pi \left( 1 + \frac{\varepsilon^{2}}{2} \right).
\end{split}    
\end{equation}
Let us now focus on $P_s(\Omega^\varepsilon_j)$. By \cite[(2.20)]{F2M3}, for $\varepsilon>0$ sufficiently small, we have
\begin{equation}\label{perOmega}
   P_s(\Omega^\varepsilon_j)=\frac{\varepsilon^2}{2}\,g(\varepsilon)+\frac{P_s(B_1)}{P(B_1)}\,h(\varepsilon),
\end{equation}
where we have set
\[
  g(\varepsilon):=\int_{\partial B_{1}}\int_{\partial B_{1}}\biggl(\int_{u(y)}^{u(x)}\int_{u(y)}^{u(x)} f_{|x-y|}(1+\varepsilon r,1+\varepsilon\varrho)\,dr\,d\varrho \biggr)\,d\mathcal{H}^1_x\,d \mathcal{H}^1_y,
\]
with
\[
  f_{\theta}(r,\varrho):=\frac{r\varrho}{(|r-\varrho|^2+r\varrho\,\theta^2)^{\frac{2+s}{2}}}, \quad \text{for all } r,\varrho,\theta>0,
\]
and
\[
  h(\varepsilon):=\int_{\partial B_1}(1+\varepsilon u(x))^{2-s}\,d\mathcal{H}^1_x.
\]
Hence, in particular,
\begin{equation}\label{stima-PsOmegaeee}
    P_s(\Omega_j^\eee)\geq \frac{1}{2}\int_{\partial B_{1}}\int_{\partial B_{1}}\biggl(\int_{1+\eee u(y)}^{1+\eee u(x)}\int_{1+\eee u(y)}^{1+\eee u(x)}\frac{r\varrho}{(|r-\varrho|^2+r\varrho\,|x-y|^2)^{\frac{2+s}{2}}}\,dr\,d\varrho\biggl)\,d\mathcal{H}^1_x\,d \mathcal{H}^1_y,
\end{equation}
Now, for $m = 0,\dots,j-1$, let us set
\[
x_m := \frac{\pi}{4j} + m \frac{\pi}{j}, \qquad
y_m := \frac{3\pi}{4j} + m \frac{\pi}{j},
\]
and $r := \frac{1}{4j}$. Define
\[
I_m := \{ x\in[0,2\pi] \,: \,|x - x_m| \le r \}, \qquad
J_m := \{ y\in[0,2\pi] \,:\, |y - y_m| \le r \},
\]
so that, in particular, $|I_m|=|J_m|=2r=\frac{1}{2j}$. Then, for any $x\in I_m$ and $y\in J_m$, we have
\begin{equation}\label{fatto1}
    |x-y| \le |x-x_m| + |x_m-y_m| + |y_m-y| \le 2r+\frac{\pi}{2j}
    =\frac{\pi/2 + 1/2}{j}<\frac{4}{j},
\end{equation}
and $u(x)=\sin(2jx)>\frac{1}{2}$ and $u(y)=\sin(2jy)<-\frac{1}{2}$, which implies
\begin{equation}\label{fatto2}
|u(x)-u(y)| = |\sin(2jx)-\sin(2jy)|
= \sin(2jx)-\sin(2jy)\geq \frac{1}{2}-\left(-\frac{1}{2}\right)=1.
\end{equation}
For any $x\in I_m$ and $y\in J_m$, defining also $$\mathcal{S}:=\left\{ (r,\varrho)\in [1+\eee u(y),1+\eee u(x)]\times[1+\eee u(y),1+\eee u(x)]:\ |r-\varrho|\leq \frac{|x-y|}{2} \right\},$$
then, for $j\in\N$ sufficiently large, $|\mathcal{S}|\geq C|x-y|$, for some universal constant $C>0$. Moreover, for any $r,\varrho \in \mathcal{S}$, we have
\begin{equation*}
    r\varrho \geq (1+\eee u(y))^2\geq (1-\eee)^2
\end{equation*}
and 
\begin{equation*}
    |r-\varrho|^2+r\varrho|x-y|^2\leq \frac{|x-y|^2}{4}+(1+\eee u(x))|x-y|^2 \leq \left(\frac{1}{4}+(1+\eee)^2\right)|x-y|^2.
\end{equation*}
Consequently,
\begin{equation}\label{stima-N}
\begin{split}
    \int_{1+\eee u(y)}^{1+\eee u(x)}\hspace{-0.3em}\int_{1+\eee u(y)}^{1+\eee u(x)}\frac{r\varrho}{(|r-\varrho|^2+r\varrho\,|x-y|^2)^{\frac{2+s}{2}}}\,dr\,d\varrho &\geq \int\hspace{-0.6em}\int_{\mathcal{S}}\frac{(1-\eee)^2}{(\frac{1}{4}+(1+\eee)^2)^{\frac{2+s}{2}}|x-y|^{2+s}}\,dr\,d\varrho \\
    &= \frac{C_0}{|x-y|^{2+s}} |\mathcal{S}|\geq\frac{C_0}{|x-y|^{1+s}}    \end{split}
\end{equation}
for some suitable constant $C_0>0$, depending only on $s$ and $\eee$, whose value is allowed to change from line to line. Finally, using \eqref{stima-PsOmegaeee}, \eqref{fatto1}, \eqref{fatto2} and \eqref{stima-N}, we obtain
\begin{equation*}
\begin{split}
    P_s(\Omega_j^\eee)&\geq \frac{1}{2}\int_{\partial B_{1}}\int_{\partial B_{1}}\biggl(\int_{1+\eee u(y)}^{1+\eee u(x)}\int_{1+\eee u(y)}^{1+\eee u(x)}\frac{r\varrho}{(|r-\varrho|^2+r\varrho\,|x-y|^2)^{\frac{2+s}{2}}}\,dr\,d\varrho\biggl)\,d\mathcal{H}^1_x\,d \mathcal{H}^1_y \\
    &\geq \sum_{m=0}^{j-1}\int_{I_m}\int_{J_{m}}\biggl(\int_{1+\eee u(y)}^{1+\eee u(x)}\int_{1+\eee u(y)}^{1+\eee u(x)}\frac{r\varrho}{(|r-\varrho|^2+r\varrho\,|x-y|^2)^{\frac{2+s}{2}}}\,dr\,d\varrho\biggl)    \,d\mathcal{H}^1_x\,d\mathcal{H}^1_y\\
    &\geq \sum_{m=0}^{j-1}\int_{I_m}\int_{J_{m}} \frac{C_0}{|x-y|^{1+s}}    \,d\mathcal{H}^1_x\,d\mathcal{H}^1_y\\    
    &\geq \left(\frac{j}{4}\right)^{1+s} C_0\,\sum_{m=0}^{j-1}\int_{I_m}\int_{J_{m}}\,d\mathcal{H}^1_x\,d\mathcal{H}^1_y 
    =C_0\left(\frac{j}{4}\right)^{1+s}j\,\frac{1}{4j^2}
    =C_0\,j^{s},
\end{split}
\end{equation*}
which implies that $P_s(\Omega_j^\varepsilon)\to +\infty$ as $j\to +\infty$.

On the other hand, by the construction of the sets $\{\Omega^\varepsilon_j\}_{j\in \mathbb{N}}$, we have $B_{1-\varepsilon}\subset \Omega_j$ for any $j\in\mathbb{N}$ and $\varepsilon$ sufficiently small. Hence,
\(
h_{m,s}(B_{1-\varepsilon})=h_{m,s}(B_1)(1-\varepsilon)^{2-s-2m}
\)
provides an upper bound (uniformly in $j$) for $h_{m,s}(\Omega_j^\varepsilon)$, that is,
\begin{equation}\label{primoboundh}
h_{m,s}(\Omega_j^\varepsilon)\leq h_{m,s}(B_1)(1-\varepsilon)^{2-s-2m}.
\end{equation}
Moreover, by \eqref{calcolovolume} and \eqref{fractional-cheeger}, we also know that 
\begin{equation}\label{secondoboundh}
h_{m,s}(B)= h_{m,s}(B_1)\left(1+\frac{\varepsilon^2}{2}\right)^{\frac{2-s-2m}{2}}\leq h_{m,s}(\Omega_j^\varepsilon),
\end{equation}
where $B$ is the ball having the same volume as $\Omega_{j}^\varepsilon$.  
Thus, $h_{m,s}(\Omega_j^\varepsilon)$ is uniformly bounded both above and below, independently of $j$, while recalling \eqref{use-def-beta}, we have
\[
\beta_s^2(\Omega^\varepsilon_j) \geq \delta_s(\Omega^\varepsilon_j) = \frac{P_s(\Omega^\varepsilon_j) - P_s(B)}{P_s(B)} \to + \infty, \quad \text{as } j\to+\infty,
\]
which, together with \eqref{primoboundh} and \eqref{secondoboundh}, contradicts \eqref{eq:failure_beta_s}.

We now turn to the second counterexample to \eqref{eq:failure_beta_s}, which is based on fractal-type sets in arbitrary dimension. To this end, we begin with the following result.   

\begin{Proposition}\label{frattali}
Let $n\geq 2$. For every $\sigma\in(0,1)$ and $c_0>0$ there exists a sequence of set $\{\Omega_M\}_{M\in\N}$ with the following properties
\begin{itemize}
    \item[(\emph{i})] There exists the $\lim_{M\to\infty}P_s(\Omega_M) $ and it holds $$\lim_{M\to\infty}P_s(\Omega_M)<\infty\quad\text{ for }\,s\in(0,\sigma)\quad\textrm{and}\quad \lim_{M\to\infty}P_s(\Omega_M)=\infty\quad\text{ for }\,s\in[\sigma,1).$$
    \item[(\emph{ii})] There exists a set $\Omega_{\infty}$ such that $$ \Omega_M\to \Omega_\infty \text{in  }L^1\, \text{ as }M\to+\infty\quad \text{and}\quad |\Omega_M|\leq |\Omega_\infty|<\infty, \text{ for any } M\in\N$$ and $$P_s(\Omega_\infty)<\infty\quad\text{ for }\,s\in(0,\sigma)\quad\textrm{and}\quad P_s(\Omega_\infty)=\infty\quad\text{ for }\,s\in[\sigma,1).$$
    \item[(\emph{iii})] For any $M\in\N$, up to translation,  $$ B_{c_0}\subset \Omega_M.$$    
\end{itemize}
\end{Proposition}
\begin{proof}To construct the family of sets $\{\Omega_M\}_{M \in \mathbb{N}}$, we employ the construction of fractal-type sets in \cite{Lom}, which we briefly check below.

Specifically, let $T_0 \subset \mathbb{R}^n$ be a bounded open set with $0 < |T_0| < \infty$ and finite perimeter, such that, up to translation, $B_{c_0} \subset T_0$. Let $a, b \in \mathbb{N}$, with $b \geq 2$ arbitrary. Consider the following family of roto-translations of  a scaling of $T_0$:
\begin{equation*}
    T_k^i := F_k^i(T_0) := \mathcal{R}_k^i \big( \lambda^{-k} T_0 \big) + x_k^i,
\end{equation*}
where $\lambda := b^{\frac{1}{n - \sigma}} > 1$, $k \in \mathbb{N}$, $1 \leq i \leq ab^{k-1}$, $\mathcal{R}_k^i \in SO(n)$, and $x_k^i \in \mathbb{R}^n$ are chosen to ensure that the sets do not overlap, i.e.,
\begin{equation*}
    |T_k^i \cap T_h^j| = 0 \quad \text{for all } (k, i) \neq (h, j) \quad \text{ and }\quad |T^i_k\cap T_0|=0\quad \text{for all } k,i.
    \end{equation*}

Then, for every $M \in \mathbb{N}$, we define
\begin{equation*}
    \Omega_M := T_0 \cup \bigcup_{k=1}^M \bigcup_{i=1}^{ab^{k-1}} T_k^i, \quad \text{and} \quad \Omega_\infty := T_0 \cup \bigcup_{k=1}^\infty \bigcup_{i=1}^{ab^{k-1}} T_k^i.
\end{equation*}
Furthermore, we assume the existence of a set $S_0 \subset \mathbb{R}^n \setminus \Omega_\infty$, with $|S_0| > 0$, such that
\begin{equation*}
    S_k^i := F_k^i(S_0) \subset \mathbb{R}^n \setminus \Omega_\infty \quad \text{for all } k \in \mathbb{N} \text{ and} \, 1 \leq i \leq ab^{k-1}.
\end{equation*}
Such an assumption is not restrictive and can be ensured by an appropriate choice of the initial set $T_0$, see \cite[Remark 3.9]{Lom}.

For any couple of disjoint sets $A,\,B\subset\mathbb R^n$, let 
$$ \mathcal L_s(A,B):=\int_A\int_B\frac{1}{|x-y|^{n+s}}\,dx\,dy.$$
Then, we obtain
\begin{equation*}
\begin{split}
P_s(\Omega_M)&=\Ll_s(\Omega_M, \Omega_M^c)=\Ll_s(T_0,\Omega_M^c)+\sum_{k=1}^M\sum_{i=1}^{ab^{k-1}}\Ll_s(T_k^i,\Omega_M^c)\\
&
\leq
\Ll_s(T_0,T_0^c)+\sum_{k=1}^M\sum_{i=1}^{ab^{k-1}}\Ll_s(T_k^i,(T_k^i)^c)
=
P_s(T_0) + \sum_{k=1}^M\sum_{i=1}^{ab^{k-1}}\Ll_s(F_k^i(T_0),F_k^i(T_0^c))\\
&
=
P_s(T_0) \left( 1+ \frac{a}{\lambda^{n-s}}\sum_{k=0}^{M-1}\Big(\frac{b}{\lambda^{n-s}}\Big)^k\right),
\end{split}\end{equation*}
and
\begin{equation*}\begin{split}
P_s(\Omega_M)&=\Ll_s(\Omega_M, \Omega_M^c)=\Ll_s(T_0,\Omega_M^c)+\sum_{k=1}^M\sum_{i=1}^{ab^{k-1}}\Ll_s(T_k^i,\Omega_M^c)\\
&
\geq
\Ll_s(T_0,S_0)+\sum_{k=1}^M\sum_{i=1}^{ab^{k-1}}\Ll_s(T_k^i,S_k^i)
=
\Ll_s(T_0,S_0) + \sum_{k=1}^M\sum_{i=1}^{ab^{k-1}}\Ll_s(F_k^i(T_0),F_k^i(S_0))\\
&
=\Ll_s(T_0,S_0)\left(1+\frac{a}{\lambda^{n-s}}\sum_{k=0}^{M-1}\Big(\frac{b}{\lambda^{n-s}}\Big)^k\right).
\end{split}\end{equation*}
Notice that, since $P(T_0)<\infty$, we have
\begin{equation*}
\Ll_s(T_0,S_0)\leq\Ll_s(T_0, T_0^c)=P_s(T_0)<\infty,
\end{equation*}
for every $s\in(0,1)$ (see, e.g., \cite[Proposition 2.1]{Lom}). This proves (\emph{i}), since $\frac{b}{\lambda^{n-s}}=b^{\frac{s-\sigma}{n-\sigma}}<1$, when $s<\sigma$.

Moreover, for any $M\in\N$, we have that $\Omega_M\subset \Omega_\infty$ and 
\begin{equation*}
    0<|\Omega_M|<|\Omega_\infty|\le |T_0|+\sum_{k=1}^\infty\sum_{i=1}^{ab^{k-1}}|T_k^i|= |T_0|\left( 1+\frac{a}{\lambda^{n}}\sum_{k=0}^\infty\Big(\frac{b}{\lambda^{n}}\Big)^k\right)<\infty
\end{equation*}
and
$$ |\Omega_M\Delta \Omega_\infty|=| \Omega_\infty\setminus \Omega_M|=\sum_{k=M+1}^{\infty}\sum_{i=1}^{ab^{k-1}}|T^i_k|=\frac{a}{\lambda^{n}}\sum_{k=M}^\infty\Big(\frac{b}{\lambda^{n}}\Big)^k \to 0\quad \text{ as }\,M\to\infty,$$
since $\frac{b}{\lambda^n}=b^{-\frac{\sigma}{n-\sigma}}<1$, as $\sigma>0$. This, together with (\emph{i}), concludes the proof of (\emph{ii}). Finally, (\emph{iii}) follows trivially by construction.
\end{proof}

Explicit examples of families of sets satisfying the assumptions of Proposition \ref{frattali} can be found in \cite{Lom}. Fixed $\sigma\in(0,1)$, let us now explain how the family of sets $\{\Omega_M\}_{M \in \mathbb{N}}$ can be used to construct a counterexample to inequality \eqref{eq:failure_beta_s}, when $s\in[\sigma,1)$.

For each $M \in \mathbb{N}$, let $B_M$ and $B_\infty$ denote the balls centered at the origin such that $|B_M| = |\Omega_M|$ and $|B_\infty| = |\Omega_\infty|$, respectively. By point (\emph{iii}) of Proposition \ref{frattali}, we have
\[
h_{m,s}(B_{c_0}) \geq h_{m,s}(\Omega_M), \quad \text{for all } M \in \mathbb{N}.
\]
Moreover, combining \eqref{fractional-cheeger} with point (\emph{ii}) of Proposition \ref{frattali}, it follows that
\[
h_{m,s}(\Omega_M) \geq h_{m,s}(B_M) \geq h_{m,s}(B_\infty), \quad \text{for all } M \in \mathbb{N}.
\]
This implies that $h_{m,s}(\Omega_M)$ is uniformly bounded both above and below, independently of $M$. Therefore, the left-hand side of \eqref{eq:failure_beta_s} remains uniformly bounded as $M \to \infty$.

On the other hand, by \eqref{use-def-beta}, we obtain
\[
\beta_s^2(\Omega_M) \geq \delta_s(\Omega_M) = \frac{P_s(\Omega_M) - P_s(B_M)}{P_s(B_M)} \geq \frac{P_s(\Omega_M) - P_s(B_\infty)}{P_s(B_\infty)}.
\]
Hence, by point (\emph{i}) of Proposition \ref{frattali}, $\beta_s^2(\Omega_M)$ diverges as $M \to \infty$. Consequently, no constant $c > 0$ can make inequality \eqref{eq:failure_beta_s} valid for all $M \in \mathbb{N}$, which completes the counterexample.

}

\section*{Acknowledgments}
The authors warmly thank Nicola Fusco for preliminary discussions on the subject.

    The Authors are member of the {\em Gruppo Nazionale per l'Analisi Ma\-te\-ma\-ti\-ca, la Probabilit\`a e le loro Applicazioni} (GNAMPA) of the {\em Istituto Nazionale di Alta Matematica} (INdAM) and are partially supported by the INdAM--GNAMPA Project \emph{"Ottimizzazione Spettrale, Geometrica e Funzionale"}, CUP: E5324001950001 and by the PRIN 2022 project 2022R537CS \emph{"$NO^3$ - Nodal Optimization, NOnlinear elliptic equations, NOnlocal geometric problems, with a focus on regularity"}, CUP: J53D23003850006, founded by the European Union - Next Generation EU.

\bibliographystyle{acm}
\bibliography{reference}

\end{document}